\documentclass[11pt]{amsart}

\usepackage{amsmath}
\usepackage{amssymb}
\usepackage{graphicx}

\newtheorem{theorem}{Theorem}[section]

\newtheorem{lemma}[theorem]{Lemma}
\newtheorem{proposition}[theorem]{Proposition}
\theoremstyle{definition}

\newtheorem{remark}[theorem]{Remark}

\numberwithin{equation}{section}

\newcommand{\R}{\mathbb{R}}
\newcommand{\dd}{\, \text{d}}
\newtheorem{hyp}[theorem]{Assumption}
\usepackage[]{algorithm2e}

\title[Numerical methods for mean-field type control problems]{Numerical methods for mean-field type optimal control problems}

\author[L. Pfeiffer]{Laurent Pfeiffer}

\address[L. Pfeiffer]{Institute for Mathematics and Scientific Computing, Karl-Franzens-Universit\"at, Heinrichstra{\ss}e 36, 8010 Graz, Austria}
\email{{\tt laurent.pfeiffer@uni-graz.at}}



\begin{document}

\begin{abstract}
In this article, two methods for solving mean-field type optimal control problems are proposed and investigated.
The two methods are iterative methods: at each iteration, a Hamilton-Jacobi-Bellman equation is solved, for a terminal condition obtained by linearizing the cost function. The terminal condition is updated by solving a Fokker-Planck equation. The first method can be seen as a gradient method and uses in an essential manner the convexity of the set of probability distributions. A convergence result for this method is provided. The second method incorporates a penalization term and provides feedback controls. We test the methods on four academic examples.
\end{abstract}

\maketitle


\section{Introduction}

\subsection{Contributions of the article}

\paragraph{Context}

This article is devoted to numerical methods for mean-field type optimal control problems. The problems under study are optimal control problems of stochastic differential equations, for which the cost function is a function of the probability distribution of the state variable at the final time. 
Cost functionals depending on a probability distribution are for instance used in risk-averse optimization: in some situations, minimizing the expectation of a random cost may lead to risky control strategies. Minimizing a function of the whole probability distribution enables then the manager to model and to take into account the risk associated with a given strategy due to the dispersion of the cost. We refer to \cite[Chapter 6]{SDR14} for examples of cost functionals in risk-averse optimization.
Mean-field type control problems are closely related to mean-field games, which have received much attention in the last years, after the publication of the seminal papers \cite{LL07} and \cite{HMC06}. For an introduction to this topic, we refer to \cite{Car12} and \cite{BFY13}.
An example is provided in \cite{LST10}, where the developpement of a new technology is modelled. Let us also mention that stochastic problems with a constraint on the probability distribution of the state variable have been recently studied: in \cite{BT13}, the final probability distribution is prescribed, in \cite{Pfe15}, the constraint is simply an expectation constraint.

The numerical methods presented in this paper are based on a resolution of the optimality conditions. When the optimal control process is a regular feedback control, the optimality conditions take the form of a coupled system of a Hamilton-Jacobi-Bellman (HJB) equation and a Fokker-Planck equation. The Fokker-Planck equation describes then the evolution of the probability distribution of the state variable and the HJB equation is an adjoint equation. The optimality conditions can be derived with the standard approach used for Pontryagin's principle with needle perturbations, see \cite[Chapter 4]{BFY13} and \cite[Proposition 3]{LP14}. In a closely related paper of the same author \cite{Pfe15a}, optimality conditions are derived for a formulation of the problem where the feasible control processes are adapted processes with respect to the Brownian motion. In this framework, the optimality conditions are formulated as follows: an optimal control process is also an optimal solution to a standard stochastic optimal control problem, where the terminal condition can be interpreted as a derivative of the cost function. A standard problem refers here to a problem where the cost function is an expectation, which can be solved by dynamic programming.

\paragraph{Description of the methods}

The methods that we propose are iterative methods, consisting in forward and backward passes. The backward pass consists in the resolution of a standard problem, by solving the corresponding HJB equation. The forward pass consists in solving the Fokker-Planck equation associated with the control process obtained in the backward pass. The proposed schemes are similar to those used to solve mean-field game problems, in so far as they are based on backward and forward passes. In the first method, convex combinations of probability measures in the reachable set are used: in this way, the scheme can be seen as a gradient method.
We provide a convergence result for the continuous-time formulation of the scheme. This formulation of the scheme provides a sequence of probability measures. We show that this sequence possesses at least one limit point satisfying the optimality conditions.
The second method looks like the fixed-point approach for mean-field games, however, a penalization term is included into the linearized standard problems to be solved. In this manner, the algorithm provides a sequence of feedback controls.
The forward-backward structure of our algorithm is usual, but the interpretation of the first method as a gradient method is new to our knowledge. The introduction of a penalizing term in the second method seems also to be new.

The state variable is discretized as a controlled Markov chain with a semi-Lagrangian scheme, proposed originally in \cite{CF95}, see also the reference \cite{Kus90} on the approximation of SDEs with Markov chains. The HJB equation can be then easily discretized: it suffices to write the dynamic programming principle associated with the controlled Markov chain. The Fokker-Planck equation is discretized by writing the Chapman-Kolmogorov equation associated with the Markov chain. In this paper, we do not analyze the discretization of the SDE from a theoretical point of view.
The numerical analysis of forward-backward systems has received much attention in the last years and is still an active field of research. In the articles \cite{ACD10,ACCD13}, a proof of convergence of an implicit finite-difference numerical scheme is provided for various mean-field games.
In \cite{CS14,CS14b}, a semi-Lagrangian scheme, similar to the one used in this article, is used for mean-field games. The convergence of the discretized solutions is proved for a state variable of dimension 1.

\paragraph{Structure of the article}

Just below, we provide the main notations of the paper and describe the problem under study.
In section \ref{section2}, we recall the main results of \cite{Pfe15a}, on which the numerical methods are based. These methods are described in section \ref{section3}, where we give a convergence result (Theorem \ref{theoConvergence}). Section \ref{section4} consists in a discussion on different kinds of cost functions which can be considered. The continuity and the differentiability of these functions are investigated. In section \ref{section5}, numerical results are shown on academic examples. The cost function of the first example involves a Wasserstein distance, in order to reach a given probability distribution. The cost functions of the next two examples take into account the standard deviation of the distribution. The last example uses the conditional value at risk.

\subsection{Formulation of the problem and assumptions}

\paragraph{General notations}

The set of probability measures on $\R^n$ is denoted by $\mathcal{P}(\R^n)$. For a function $\phi:\R^n \rightarrow \R$, its integral (if well-defined) with respect to the measure $m \in \mathcal{P}(\R^n)$ is denoted by
\begin{equation*}
\int_{\R^n} \phi(x) \dd m(x) \quad \text{or} \quad
\int_{\R^n} \phi \dd m.
\end{equation*}
Given two measures $m_1$ and $m_2 \in \mathcal{P}(\R^n)$, we denote:
\begin{equation*}
\int_{\R^n} \phi(x) \dd (m_2(x)- m_1(x)) := \int_{\R^n} \! \phi(x) \dd m_2(x) - \int_{\R^n} \! \phi(x) \dd m_1(x).
\end{equation*}

The probability distribution of a given random variable $X$ with values in $\R^n$ is denoted by $\mathcal{L}(X) \in \mathcal{P}(\R^n)$. If $m= \mathcal{L}(X) \in \mathcal{P}(\R^n)$, then for any continuous and bounded function $\phi:\R^n \rightarrow \R$,
\begin{equation*}
\mathbb{E}\big[ \phi(X) \big]= \int_{\R^n} \phi \dd m.
\end{equation*}
We also denote by $\sigma(X)$ the $\sigma$-algebra generated by $X$.

For $p \geq 1$, we denote by $\mathcal{P}_p(\R^n)$ the set of probability measures having a finite $p$-th moment:
\begin{equation*}
\mathcal{P}_p(\R^n) := \Big\{
m \in \mathcal{P}(\R^n) \,\big|\, \int_{\R^n} |x|^p \dd m(x) < + \infty
\Big\}.
\end{equation*}
We equip $\mathcal{P}_p(\R^n)$ with the Wasserstein distance $d_p$, see the definition and the dual representation of $d_p$ for $p=1$ in section \ref{section21}.

For all $R\geq 0$, we define:
\begin{equation} \label{eqDefBpR}
\bar{B}_p(R):= \Big\{ m \in \mathcal{P}_p(\R^n) \,|\, \int_{\R^n} |x|^p \dd m(x) \leq R \Big\}.
\end{equation}

The open (resp.\@ closed) ball of radius $r \geq 0$ and center $0$ is denoted by $B_r$ (resp.\@ $\bar{B}_r$), its complement by $B_r^\text{c}$ (resp.\@ $\bar{B}_r^\text{c}$). The set of real-valued Lipschitz continuous functions with modulus 1 defined on $\R^n$ is denoted by $1-\text{Lip}(\R^n)$.

For a given $p \geq 1$, a function $\phi:\R^n \rightarrow \R^n$ is said to be dominated by $|x|^p$ if for all $\varepsilon>0$, there exists $r>0$ such that for all $x \in B_r^{\text{c}}$,
\begin{equation} \label{eqContPropR}
|\phi(x)| \leq \varepsilon |x|^p.
\end{equation}

\paragraph{Controlled SDEs}

We fix $T>0$ and a standard Brownian motion $(W_t)_{t \in [0,T]}$ of dimension $d$.
For all $t \in [0,T]$, we denote  by $\mathcal{F}_{t}$ the $\sigma$-algebra generated by $(W_s)_{s \in [0,t]}$. 

Let $U$ be a compact subset of $\R^n$. For a given random variable $Y_0$ independent of $\mathcal{F}_{0,T}$ with values in $\R^n$, we denote by $\mathcal{U}_0(Y_0)$ the set of control processes $(u_t)_{t \in [0,T]}$ with values in $U$ which are such that for all $t$, $u_t$ is $(\sigma(Y_0) \times \mathcal{F}_{0,t})$-measurable.

The drift $b:\R^n \times U \rightarrow \R^n$ and the volatility $\sigma:\R^n \times U \rightarrow \R^{n \times d}$ are given. For all $u \in \mathcal{U}_0(Y_0)$, we denote by $\big( X_t^{0,Y_0,u} \big)_{t \in [0,T]}$ the solution to the following SDE:
\begin{equation} \label{eqGenSDE}
\dd X_t^{0,Y_0,u}= b(X_t^{0,Y_0,u},u_t) \dd t + \sigma(X_t^{0,Y_0,u},u_t) \dd W_t, \ \ \forall t \in [0,T], \quad X_0^{0,Y_0,u}= Y_0.
\end{equation}
The well-posedness of this SDE is ensured by the Assumption \ref{hypLipschitzCoeff} \cite[Section 5]{Oks03} below. We also denote by $m_t^{0,Y_0,u}$ the probability distribution of $X_t^{0,Y_0,u}$:
\begin{equation*}
m_t^{0,Y_0,u}= \mathcal{L}(X_t^{0,Y_0,u}).
\end{equation*}

All along the article, we assume that the following assumption holds true.
\begin{hyp} \label{hypLipschitzCoeff} There exists $L>0$ such that for all $x,y \in \R^n$, for all $u, v \in U$,
\begin{align*}
& |b(x,u)| + |\sigma(x,u)| \leq L(1+|x| + |u|), \\
& |b(x,u)-b(y,v)| + |\sigma(x,u)-\sigma(y,v)|  \leq L(|y-x|+ |v-u|).
\end{align*}
\end{hyp}

\paragraph{Formulation of the problem}

We fix an initial condition $Y_0$ (independent of $\mathcal{F}_{0,T}$) and $p \geq 2$ such that $\mathcal{L}(Y_0) \in \mathcal{P}_p(\R^n)$.
Let $\chi$ be a real-valued function defined on $\mathcal{P}_p(\R^n)$. We aim at studying the following problem:
\begin{equation*} \tag{$P$} \label{eqPb}
\inf_{u \in \mathcal{U}_0(Y_0)} \ \chi(m^{0,Y_0,u}_T).
\end{equation*}

Throughout the article, we assume that the next two assumptions (dealing with the continuity and the differentiability of $\chi$) are satisfied.

\begin{hyp} \label{hypContinuity}
The mapping $\chi$ is continuous for the $d_1$-distance.
\end{hyp}

In order to state optimality conditions, we will need a notion of derivative for the mapping $\chi$. There are different ways to define the derivative of $\chi$ and we refer to \cite[Section 6]{Car12} for a discussion on this topic. Denoting by $\mathcal{M}(\R^n)$ the set of finite signed measures on $\R^n$, we define:
\begin{equation*}
\widehat{\mathcal{M}}_p(\R^n)
= \Big\{
m \in \mathcal{M}(\R^n) \,\big|\, \int_{\R^n} |x|^p \dd |m|(x) < + \infty, \ \int_{\R^n} 1 \dd m(x)= 0
\Big\}.
\end{equation*}

\begin{hyp} \label{hypDiff}
The mapping $\chi$ is directionally differentiable in the following sense: for all $m_1$ in $\mathcal{P}_p(\R^n)$, there exists a linear form $D\chi(m_1)$ on $\widehat{\mathcal{M}}_p(\R^n)$ such that for all $m_2$ in $\mathcal{P}(\R^n)$, for all $\theta \in [0,1]$,
\begin{equation} \label{eqDirectionalDerivative}
\chi \big( (1-\theta)m_1+\theta m_2 \big) = \chi(m_1) + \theta D\chi(m_1)(m_2-m_1) + o(\theta).
\end{equation}
Moreover, we assume that the linear form can be identified with a continuous function denoted by $x \in \R^n \mapsto D \chi(m_1,x)$ which is dominated by $|x|^p$.
\end{hyp}

Under Assumption \ref{hypDiff}, equation \eqref{eqDirectionalDerivative} reads:
\begin{equation} \label{eqDirectionalDerivative2}
\chi \big( (1-\theta)m_1+\theta m_2 \big) = \chi(m_1) + \theta \Big[ \int_{\R^n} D\chi(m_1,x) \dd(m_2(x)-m_1(x)) \Big] + o(\theta).
\end{equation}
Observe that at a given value of $m_1$, the directional derivative of $D\chi(m_1,\cdot)$ is uniquely defined ``up to a constant": equation \eqref{eqDirectionalDerivative2} remains true if a constant is added to $D\chi(m_1,\cdot)$. The uniqueness derives directly from the identity:
\begin{equation*}
D\chi(m_1,y)= \lim_{\theta \to 0} \frac{\chi((1-\theta)m_1 + \theta \delta_y)-\chi(m_1)}{\theta} + \int_{\R^n} D\chi(m_1,x) \dd m_1(x),
\end{equation*}
where $\delta_y$ is the Dirac measure centered at $y$.

\section{Analysis of the problem} \label{section2}

\subsection{Technical results} \label{section21}

\paragraph{Convexity of the closure of the reachable set}

We denote by $\mathcal{R}(0,Y_0)$ the set of reachable probability measures at time $T$:
\begin{equation} \label{eqDefReachableSet}
\mathcal{R}(0,Y_0)= \big\{ m_T^{0,Y_0,u} \,|\, u \in \mathcal{U}_0(Y_0) \}.
\end{equation}
We denote by $\text{cl}\big( \mathcal{R}(0,Y_0) \big)$ its closure for the $d_1$-distance. We recall Lemma 6 of \cite{Pfe15a}.

\begin{lemma} \label{lemmaConvexity}
The closure of the set of reachable probability measures 
is convex.
\end{lemma}

\paragraph{Elements on optimal transportation}

We first recall the definition of the Wasserstein distance denoted by $d_p$ (for $p \geq 1$) in this article. For all $m_1$ and $m_2$ in $\mathcal{P}_p(\R^n)$,
\begin{equation} \label{eqDefWasserstein}
d_p(m_1,m_2)= \Big[ \inf_{\pi \in \Pi(m_1,m_2)} \int_{\R^n \times \R^n} |y-x|^p \dd \pi(x,y) \Big]^{1/p},
\end{equation}
$\Pi(m_1,m_2)$ being the set of transportation mappings from $m_1$ to $m_2$ defined as:
\begin{equation*}
\Bigg\{ \pi \in \mathcal{P}(\R^{2n}) \,|\, \Big\{ \begin{array}{l} \pi(A\times \R^n)= m_1(A),\\ \pi(\R^n \times A)= m_2(A),\end{array} \, \text{for all measurable $A \subset \R^n$} \Bigg\}.
\end{equation*}
By H\"older's inequality, for all $1 \leq p \leq p'$, $\mathcal{P}_{p'}(\R^n) \subset \mathcal{P}_p(\R^n)$. Moreover, for all $1 \leq p \leq p'$ and for all $m$ and $m'$ in $\mathcal{P}_{p'}(\R^n)$,
\begin{equation*}
d_p(m,m') \leq d_{p'}(m,m').
\end{equation*}
Note that $\mathcal{P}_p(\R^n)$ equipped with $d_p$ is complete and separable \cite[Theorem 6.18]{Vil09}. Note also the dual representation of $d_1$ \cite[Remark 6.5]{Vil09}: for all $m_1$, $m_2 \in \mathcal{P}_1(\R^n)$,
\begin{equation} \label{eqDualWasserstein}
d_1(m_1,m_2)= \sup_{\phi \in 1-\text{Lip}(\R^n)} \int_{\R^n} \phi \dd(m_2- m_1),
\end{equation}
where $1-\text{Lip}(\R^n)$ is the set of Lipschitz function with modulus 1.

The following lemma is a classical result, see for example \cite[Lemma 5.7]{Car12}.

\begin{lemma} \label{lemmaCompactnessProperty}
For all $p>1$ and $R\geq 0$, the subset $\bar{B}_p(R)$ of $\mathcal{P}_1(\R^n)$ (defined in \eqref{eqDefBpR}) is compact for the $d_1$-distance.
\end{lemma}

A proof of the following lemma can be found in the appendix of \cite{Pfe15a}.

\begin{lemma} \label{lemmaContinuityDominatedCost}
Let $p>1$, $\phi:\R^n \rightarrow \R$ be dominated by $|x|^p$ (in the sense of \eqref{eqContPropR}).
Then, for all $R \geq 0$, the following mapping: $m \in \bar{B}_p(R) \mapsto \int_{\R^n} \phi(x) \dd m(x)$
is continuous for the $d_1$-distance.
\end{lemma}

\subsection{Optimality conditions}

In this section, we give a maximum principle for problem \ref{eqPb} (defined in the introduction, page \pageref{eqPb}). We start by recalling the approach by dynamic programming for a \emph{linear} cost function of the form:
\begin{equation*}
\chi(m) = \int_{\R^n} \phi(x) \dd m(x),
\end{equation*}
where $\phi:\R^n \rightarrow \R$ is continuous and dominated by $|x|^p$.
The term ``linear" refers here to the following property: for all $m_1$, $m_2 \in \mathcal{P}_p(\R^n)$, for all $\theta \in [0,1]$,
\begin{equation*}
\chi \big( \theta m_1 + (1-\theta) m_2 \big) = \theta \chi(m_1) + (1-\theta) \chi(m_2).
\end{equation*}
In this case, the derivative introduced in Assumption \ref{hypDiff} is equal to $\phi$ for all $m$.

We denote by \ref{eqLinearizedPb} the following problem:
\begin{equation*} \label{eqLinearizedPb} \tag{$P(\phi)$}
\inf_{u \in \mathcal{U}_0(Y_0)} \, \mathbb{E}\big[ \phi(X_T^{0,Y_0,u}) \big].
\end{equation*}
We call such a problem \emph{standard} problem.
We set $a(x,u)= \sigma(x,u) \sigma(x,u)^\text{t}$ and define the unminimized Hamiltonian $h(u,x,p,Q)$ and the Hamiltonian $H(x,p,Q)$ by
\begin{align*}
h(u,x,p,Q)=\ & pb(u,x) + \frac{1}{2} \text{tr}( a(x,u)Q ), \\
H(x,p,Q) =\ & \inf_{u \in U} \ h(u,x,p,Q),
\end{align*}
where $p$ is a row vector of size $n$ and $Q$ a symmetric matrix of size $n$.
As is well-known (see for example the books \cite{BCD08,FS93}),
the standard problem \ref{eqLinearizedPb} can be solved by dynamic programming. We introduce the value function $V:[0,T] \times \R^n \rightarrow \R$, defined by:
\begin{align*}
& V(t,x) = \inf_{u \in \mathcal{U}_t} \, \mathbb{E} \big[ \phi(X_T^{t,x,u}) \big], \\
& \qquad \text{s.\@ t.: }
\begin{cases} \begin{array}{rl}
\dd X_s^{t,x,u}= & f(X_s^{t,x,u},u_s) \dd s + \sigma(X_s^{t,x,u},u_s) \dd W_s, \\
X_t^{t,x,u}= & x,
\end{array}
\end{cases}
\end{align*}
where $(W_s)_{s \in [t,T]}$ is a standard Brownian motion and $\mathcal{U}_t$ the set of adapted processes with respect to the filtration generated by $(W_s)_{s \in [t,T]}$.
It is the viscosity solution to the following Hamilton-Jacobi-Bellman (HJB) equation:
\begin{equation} \label{eqAdjoint}
-\partial_t V(t,x)= H(x,\partial_x V(t,x),\partial_{xx} V(t,x)), \quad V(T,x)= \phi(x).
\end{equation}
An optimal solution $\bar{u}$ to \ref{eqLinearizedPb} is then such that for a.a.\@ $t$, $\bar{u}_t$ minimizes almost surely $h\big( \cdot, X_t^{0,Y_0,\bar{u}},\partial_x V(t,X_t^{0,Y_0,\bar{u}}),\partial_{xx} V(t,X_t^{0,Y_0,\bar{u}})\big)$, if $V$ is sufficiently regular.
The dynamic programming principle associated with $V$ states also that
\begin{equation*}
\mathbb{E}\big[ V(t,X_t^{0,Y_0,\bar{u}}) \big] \ \Big( = \int_{\R^n} V(t,x) \dd m_t^{0,Y_0,\bar{u}}(x) \Big)
\end{equation*}
is independent of $t \in [0,T]$ and equal to the value of \ref{eqLinearizedPb}.

The following theorem was proved in \cite{Pfe15a}. It is a maximum principle for problem \ref{eqPb}. For the sake of completeness, we recall here a proof.

\begin{theorem} \label{theoOptiCond}
Let $\bar{u} \in \mathcal{U}_0(Y_0)$ be an optimal solution to problem \ref{eqPb}. Then, $\bar{u}$ is a solution to the standard problem \ref{eqLinearizedPb}, where $\phi(\cdot)= D\chi(\bar{m},\cdot)$, with $\bar{m}=m_T^{0,Y_0,\bar{u}}$.
\end{theorem}

In the sequel, we call problem \ref{eqLinearizedPb} \emph{linearized problem} when $\phi(\cdot)= D\chi(\bar{m},\cdot)$, in order to emphasize its connection with problem \ref{eqPb}.

\begin{proof}

Using the continuity of $\chi$ of the $d_1$-distance, we obtain that $\bar{m}$ is optimal on cl$(\mathcal{R}(0,Y_0))$, which is by Lemma \ref{lemmaConvexity} a convex set. Therefore, for all $u \in \mathcal{U}_0(Y_0)$, for all $\theta \in [0,1]$,
\begin{align}
0 \leq \ & \chi \big( (1-\theta)\bar{m} + \theta m_T^{0,Y_0,u} \big)-\chi(\bar{m}) \notag \\
 = \ & \theta D\chi(\bar{m})(m_T^{0,Y_0,u}-\bar{m}) +o(\theta) \notag \\ 
 = \ & \theta \mathbb{E} \big[ D\chi \big( \bar{m},X_T^{0,Y_0,u} \big) - D\chi \big( \bar{m},X_T^{0,Y_0,\bar{u}} \big) \big] + o(\theta). \label{eqDerivationOC}
\end{align}
Therefore, we obtain:
\begin{equation*}
\mathbb{E} \big[ D\chi \big( \bar{m},X_T^{0,Y_0,u} \big) \big] \geq \mathbb{E}  \big[ D\chi \big( \bar{m},X_T^{0,Y_0,\bar{u}} \big) \big].
\end{equation*}
The theorem is proved.
\end{proof}

We are not able to prove the existence of an optimal solution $\bar{u} \in \mathcal{U}_0(Y_0)$ in general. However, one can derive from the continuity of $\chi$ and the compactness of $\text{cl}(\mathcal{R}(0,Y_0))$ the existence of an optimal solution $\bar{m}$ to the problem: $\inf_{m \in \text{cl}(\mathcal{R}(0,Y_0))} \chi(m)$. The probability distribution $\bar{m}$ is then a solution to $\inf_{m \in \text{cl}(\mathcal{R}(0,Y_0))} D \chi(\bar{m})m$.

Here, it is not possible in general to compute directly the value function in order to obtain a characterization of the optimal solution (as we would do to solve problem \ref{eqLinearizedPb}), since the terminal condition $D\chi(\bar{m},\cdot)$ itself depends on the optimal control.

The following lemma explains the role of the value function associated with the linearized problem when $\chi$ is convex. Note that in this case, the necessary condition of Theorem \ref{theoOptiCond} is also a sufficient condition.

\begin{lemma} \label{lemmaEstimOptimality}
Assume that $\chi$ is convex on $\mathcal{P}_p(\R^n)$, that is to say, for all $\theta \in [0,1]$, for all $m_1$ and $m_2 \in \mathcal{P}_p(\R^n)$,
\begin{equation} \label{eqConvexityChi}
\chi(\theta m_1 + (1-\theta) m_2) \leq \theta \chi(m_1) + (1-\theta) \chi(m_2).
\end{equation}
Then, for all $\bar{m} \in \text{cl}(\mathcal{R}(0,Y_0))$, the following upper estimate holds:
\begin{equation} \label{eqEstimOptimality}
\chi(\bar{m})- \Big( \inf_{m \in \text{cl}(\mathcal{R}(0,Y_0))} \chi(m) \Big) \leq D\chi(\bar{m})\bar{m} - \text{Val}(P(\phi)),
\end{equation}
where $\phi= D\chi(\bar{m},\cdot)$.
In particular, if a control process $\bar{u}$ is a solution to $P(\phi)$, with $\phi= D\chi(m_T^{0,Y_0,\bar{u}})$, then $\bar{u}$ is an optimal solution to \ref{eqPb}.
\end{lemma}

\begin{proof}
Since $\chi$ is convex, the following inequality holds true for all $m \in \mathcal{P}_p(\R^n)$: 
\begin{equation*}
\chi(m)-\chi(\bar{m}) \geq D\chi(\bar{m})(m-\bar{m}).
\end{equation*}
The lemma follows directly, minimizing both sides of the last inequality.
\end{proof}

\section{Numerical method} \label{section3}

\subsection{Continuous numerical method}

In this subsection, we investigate a continuous-time algorithm. Note that the corresponding discretized algorithm is the algorithm 1 studied in section \ref{subsectionAlgo}. The procedure generates a sequence $(m^\ell)_{\ell \in \mathbb{N}}$ in $\text{cl}(\mathcal{R}(0,Y_0))$. At iteration $\ell$, the linearized problem \eqref{eqLinearizedPb} is solved with $\phi= D\chi(m^\ell)$. The probability measure corresponding to an optimal solution, denoted by $\tilde{m}^{\ell +1}$, is used as a descent direction. The next iterate is chosen as the optimal probability measure on the interval $[m^\ell,\tilde{m}^{\ell +1}]$. The parameter $\varepsilon_\ell$ is non-negative and equal to 0 if and only if $m^\ell$ satisfied the optimality condition provided by Theorem \ref{theoOptiCond}. It measures the lack of optimality when $\chi$ is convex, as proved in Lemma \ref{lemmaEstimOptimality}.

\vspace{2mm}

\begin{algorithm*}[H]
Choose $m_0 \in \text{cl}(\mathcal{R}(0,Y_0))$ and set $\ell=0$\;
\For{$\ell=0,...$}{
	Compute a solution $\tilde{m}^{\ell+1}$ to:
	\begin{equation*}
	\inf_{m \in \text{cl}(\mathcal{R}(0,Y_0))} D\chi(m^\ell)m;
	\end{equation*}
	Set $\varepsilon_\ell= D\chi(m^\ell)(m^\ell-\tilde{m}^{\ell +1}) \geq 0$\;
	Compute a solution $\theta^{\ell+1}$ to:
	\begin{equation*}
	\inf_{\theta \in [0,1]} \, \chi \big( (1-\theta)m^\ell + \theta \tilde{m}^{\ell +1} \big);
	\end{equation*}
	Set $m^{\ell+1}= (1-\theta^{\ell +1})m^\ell + \theta^{\ell +1} \tilde{m}^{\ell +1}$\;
	Set $\ell= \ell+1$ \;
}
\phantom{.}
\centering
\textbf{Continuous-time algorithm:} gradient descent

\phantom{.}

\end{algorithm*}

The convergence result is based on the following assumption.

\begin{hyp} \label{hypLipDerivative}
There exists a constant $K \geq 0$ such that for all $m_1$, $m_2$, $m_3$, and $m_4$ in $\text{cl}(\mathcal{R}(0,Y_0))$, the following estimates hold:
\begin{align}
& \big( D\chi(m_2)-D\chi(m_1) \big) (m_2-m_1) \leq K d_1(m_1,m_2)^2, \label{hypLipDerivative1} \\
& \big( D\chi(m_2)-D\chi(m_1) \big) (m_4-m_3) \leq K d_1(m_1,m_2). \label{hypLipDerivative2}
\end{align}
\end{hyp}

Equation \eqref{hypLipDerivative2} is a Lipschitz-continuity property for the derivative $D\chi$. Equation \eqref{hypLipDerivative1} actually derives from \eqref{hypLipDerivative1} and is a semi-concavity property.
We provide a general class of functions for which this assumption is satisfied in section \ref{section41}, Lemma \ref{lemmaAssSatisfied}.
Note also that by \eqref{hypLipDerivative1}, for all $m_1$ and $m_2$ in $\text{cl}(\mathcal{R}(0,Y_0))$, for all $\theta \in [0,1]$,
\begin{equation}
\big( D\chi((1-\theta) m_1 + \theta m_2)-D\chi(m_1) \big) (m_2-m_1) \leq \theta K d_1(m_1,m_2)^2 \label{hypLipDerivative3}
\end{equation}

In the following theorem, we prove the existence of a limit point (to the sequence $(m^\ell)_{\ell \in \mathbb{N}}$) satisfying the optimality conditions provided by Theorem \ref{theoOptiCond}.
The used arguments are adapted from classical proofs of convergence for gradient methods, see for example \cite[Theorem 2.11]{BGLS06}. The main idea consists in finding an upper estimate of the decay of $\chi(m^\ell)$ at each iteration, derived from the Lipschitz-continuity property \eqref{hypLipDerivative1}. 

\begin{theorem} \label{theoConvergence}
Assume that Assumption \ref{hypLipDerivative} holds.
Then, the sequence $(m^\ell)_{\ell \in \mathbb{N}}$ has at least one limit point $\bar{m}$ such that:
\begin{equation} \label{eqOptiCondLim}
D\chi(\bar{m}) \bar{m} = \inf_{m \in \text{cl}(\mathcal{R}(0,Y_0))} D\chi(\bar{m})m.
\end{equation}
Moreover, $\chi(m^\ell) \rightarrow \chi(\bar{m})$ and $\varepsilon_\ell \rightarrow 0$.
\end{theorem}

\begin{proof}
\emph{Step 1:} estimate of the decay at iteration $\ell$. By \eqref{hypLipDerivative3}, for all $\theta \in [0,1]$,
\begin{equation} \label{eqEstimeCv1}
D\chi \big( (1-\theta) m^\ell + \theta \tilde{m}^{\ell +1} \big)(\tilde{m}^{\ell +1}-m^{\ell}) \leq - \varepsilon_{\ell} + \theta K d_1(m^{\ell},\tilde{m}^{\ell +1})^2.
\end{equation}
We define then:
\begin{equation*}
\theta_0^{\ell +1}= \min \Big( \frac{\varepsilon_{\ell}}{K d_1(m^\ell,\tilde{m}^{\ell + 1})^2},1 \Big).
\end{equation*}
Since $\theta^{\ell +1}$ is optimal, it holds:
\begin{align}
\chi(m^{\ell +1})
\leq \ & \chi \big( (1-\theta_0^{\ell +1}) m^\ell + \theta_0^{\ell +1} \tilde{m}^{\ell +1} \big) \notag \\
\leq \ & \chi(m^\ell) + \int_0^{\theta_0^{\ell +1}} D\chi((1-\theta)m^\ell + \theta \tilde{m}^{\ell +1})(\tilde{m}^{\ell + 1}-m^\ell) \dd \theta \notag \\
\leq \ & \chi(m^\ell) + \int_0^{\theta_0^{\ell +1}} \!\! \big[ -\varepsilon_{\ell} + \theta K d_1(m^\ell,\tilde{m}^{\ell +1})^2 \big] \dd \theta \notag \\
\leq \ & \chi(m^\ell) - \varepsilon_{\ell} \theta_0^{\ell +1} + \frac{1}{2} (\theta_0^{\ell +1})^2 K d_1(m^\ell,\tilde{m}^{\ell +1})^2. \label{eqEstimeCv2}
\end{align}
We distinguish now two cases.
\begin{itemize}
\item If $\theta_0^{\ell+1}= \frac{\varepsilon_{\ell}}{K d_1(m^\ell,\tilde{m}^{\ell + 1})^2}$, then
by \eqref{eqEstimeCv2},
\begin{equation} \label{eqEstimeCv3}
\chi(m^{\ell +1}) - \chi(m^{\ell}) \leq - \frac{\varepsilon_{\ell}^2}{2K d_1(m^{\ell},\tilde{m}^{\ell +1})^2}
\leq - \frac{\varepsilon_{\ell}^2}{2KD^2},
\end{equation}
where $D$ is the diameter of $\text{cl}(\mathcal{R}(0,Y_0))$:
\begin{equation} \label{eqDiameter}
D= \sup_{m_1,m_2 \in \text{cl}(\mathcal{R}(0,Y_0))} d_1(m_2,m_1) < +\infty.
\end{equation}
The diameter is finite, since $\text{cl}(\mathcal{R}(0,Y_0))$ is compact, by Lemma \ref{lemmaCompactnessProperty}.
\item If $\theta_0^{\ell+1}= 1$, then $Kd_1(m^\ell,\tilde{m}^{\ell +1})^2 \leq \varepsilon_{\ell}$ and therefore, by \eqref{eqEstimeCv2},
\begin{equation} \label{eqEstimeCv4}
\chi(m^{\ell +1}) - \chi(m^{\ell}) \leq - \varepsilon_{\ell} + \frac{1}{2} K d_1(m^\ell, \tilde{m}^{\ell +1})^2 \leq -\frac{1}{2} \varepsilon_{\ell}.
\end{equation}
\end{itemize}

\emph{Step 2:} conclusion. The existence of a converging subsequence is a consequence of the compactness of $\text{cl}(\mathcal{R}(0,Y_0))$ proved in Lemma \ref{lemmaCompactnessProperty}. Since $\chi(m^\ell)$ is decreasing and since $\chi$ is continuous for the $d_1$-distance, $\chi(m^\ell) \rightarrow \chi(\bar{m})$. 
Therefore, $\chi(m^{\ell})-\chi(m^{\ell + 1}) \rightarrow 0$ and as a consequence of \eqref{eqEstimeCv3} and \eqref{eqEstimeCv4}:
\begin{equation*}
0 \leq \varepsilon_{\ell} \leq \max \big[ 2(\chi(m^\ell)-\chi(m^{\ell +1})),
\big( 2KD^2(\chi(m^\ell)-\chi(m^{\ell+1}))\big)^{1/2}
\big] \rightarrow 0.
\end{equation*}
Finally, we prove \eqref{eqOptiCondLim}. Let $m \in \text{cl}(\mathcal{R}(0,Y_0))$. Observe first that:
\begin{align*}
& D\chi(m^\ell)(m-m^\ell)-D\chi(\bar{m})(m-\bar{m})
= \\
& \qquad \big( D\chi(m^\ell)-D\chi(\bar{m}) \big)(m-m^\ell) + D\chi(\bar{m})(\bar{m}-m^\ell) \rightarrow 0.
\end{align*}
Indeed, the first term of the r.h.s.\@ converges to 0 by \eqref{hypLipDerivative2}; the second term also converges to 0 by Lemma \ref{lemmaContinuityDominatedCost}. Then:
\begin{equation*}
0 \leq \lim -\varepsilon_\ell \leq \lim D\chi(m^\ell)(m-m^\ell) = D\chi(\bar{m})(m-\bar{m}),
\end{equation*}
which concludes the proof.
\end{proof}

\subsection{Discretization of the control process and state equation}

We discretize the SDE as a controlled Markov chain on a finite subset of $\R^n$ with a semi-Lagrangian scheme, as in \cite{CF95}.

We first discretize the state space. Let $N_X \in \mathbb{N} \backslash \{ 0 \}$ and let $S=\{ x_i \,|\, i=1,...,N_X \}$ be a set of $N_X$ points in $\R^n$. We discretize the state equation so that the state variable takes only values in $S$. We set $\mathcal{X}= \text{conv}(S)$ and denote by $P_{\mathcal{X}}$ the orthogonal projection on $\mathcal{X}$. We denote by $\mathcal{P}(S)$ the set of probability measures in $S$, that we identify with:
$\big\{ \alpha \in \R_+^{N_X} \,|\, \sum_{k=1}^{N_X} \alpha_k = 1 \big\}.
$

For all $q \in \mathbb{N}$ and for all families $(x_i)_{i=1,...,q+1}$ in $\R^n$, we say that the convex envelope $\mathcal{T}$ of the set $\{ x_1,...,x_{q+1} \}$ is a non-degenerate $q$-simplex if the family $(x_2-x_1,...,$ $x_{q+1}-x_1)$ is linearly independent. The points $x_1$,...,$x_{q+1}$ are then called vertices (of the simplex $\mathcal{T}$).
Let $\mathcal{T}=(\mathcal{T}_i)_{i=1,...,N}$ be a family of $n$-simplices with vertices in $\{ x_1,...,x_{N_X} \}$. This family is called triangulation if $\mathcal{X}= \cup_{i=1}^N \mathcal{T}_i$ and if for all $1 \leq i < j \leq N$, $\mathcal{T}_i \cap \mathcal{T}_j$ is either empty or is a non-degenerate $q$-simplex with $q < n$ and with all its vertices in the intersection of the set of vertices of $\mathcal{T}_i$ and $\mathcal{T}_j$.

Henceforth we assume that a triangulation of $S$ is given. It is easy to check that for all $x \in \mathcal{X}$, there exists a unique vector $(\alpha_\ell(x))_{\ell=1,...,N_X}$ in $\mathcal{P}(S)$ such that $x= \sum_{\ell=1}^{N_X} \alpha_\ell(x) x_\ell$,
and such that there exists $r \in \{ 1,...,N \}$ for which $\{ x_\ell \,|\, \alpha_\ell(x) > 0 \} \subset \mathcal{T}_r$.

We introduce now a discretization in time. Let $N_T \in \mathbb{N} \backslash \{ 0 \}$. We set: $\delta t = T/N_T$ and define for all $x \in \R^n$, for all $u \in U$, and for all $i=1,...,2d$:
\begin{equation*}
F(x,u,i)=
\begin{cases}
\begin{array}{ll}
x + f(x,u)\delta t + \sigma_i(x,u) \sqrt{d\delta t} & \text{for $i=1,...,d$} \\
x + f(x,u)\delta t - \sigma_i(x,u) \sqrt{d\delta t} & \text{for $i=d+1,...,2d$}.
\end{array}
\end{cases}
\end{equation*}
Here, $\sigma_i(x,u)$ stands for the $i$-th column of $\sigma(x,u)$. The underlying idea is the following: with probability $1/(2d)$, the variation of the Brownian motion $\text{d} W_t$ is equal to one of the vector of the canonical basis of $\R^d$ multiplied by $\sqrt{d\delta t}$ or $-\sqrt{d\delta t}$.

We can now combine the discretization in time and space. The index $j$ is a time index and the index $k$ is a space index.
The considered control processes are only feedback controls, that is to say, elements of $U^{N_T \times N_X}$. At time $j$, the control process $u \in U^{N_T \times N_X}$ is seen as a function $u_j \in U^{N_X}$ of $1,...,N_X$. For $v \in U$, we set:
\begin{equation*}
P_{v}(k,k')= \frac{1}{2d} \sum_{i=1}^{2d} \alpha_{k'}(P_{\mathcal{X}}(F(x_k,v,i))) \geq 0,
\end{equation*}
where $P_{\mathcal{X}}$ is the projection on $\mathcal{X}$.
Note that $\sum_{k'=1}^{N_X} P_v(k,k')=1$.
For $u_j \in U^{N_X}$, we denote:
\begin{equation*}
P_{u_j}(k,k')= \frac{1}{2d} \sum_{i=1}^{2d} \alpha_{k'}(P_{\mathcal{X}}(F(x_k,u_j(k),i))),
\end{equation*}
The distinction between the two notations will be clear from the context. Note that $P_{u_j}$ is a stochastic matrix.
For $u \in U^{N_T \times N_X}$ the discretized process is denoted $(X_j^u)_{j=0,...,N_T}$ and the corresponding probability measure $(m_j^u)_{j=0,...,N_T}$. The state equation is now given by: $\forall k,k'=1,...,N_X$, $\forall j=0,...,N_T-1$,
\begin{equation}
\mathbb{P} \big[ X_{j+1}= x_{k'} \,|\, X_j= x_k \big] = P_{u_j}(k,k'),
\end{equation}
and the dynamics of its probability distribution in $\mathcal{P}(S)$ is simply given by
$
m_{j+1} = P_{u_j}^{\text{t}} m_j.
$
This equation is nothing but the Chapman-Kolmogorov equation for Markov chains.
Finally, we consider the following discretization of the probability distribution $\mathcal{L}(Y_0)$: 
\begin{equation*}
m_0(k)= \mathbb{E}\big[ \alpha_k (P_\mathcal{X}(Y_0)) \big].
\end{equation*}
Note that in \cite{CF95}, a convergence result for the associated value function is provided, but not for the Markov chain itself.

\paragraph{On the evaluation of $\chi$}

The cost functional $\chi$ needs only to be evaluated on $\mathcal{P}(S)$, which is the set of probability measures on $S=(x_i)_{i=1,...,N_X}$. Observe that for all $m_1$, $m_2 \in \mathcal{P}(S)$, which can be identified with elements of $\R^{N_X}$, it holds:
\begin{align*}
\ & \chi((1-\theta)m_1 + \theta m_2)-\chi(m_1) \\
= \ & \theta \int_{\R^n} D\chi(m_1,x) \dd \big( \textstyle{\sum_{k=1}^{N_X}} (m_2(k)-m_1(k)) \delta_{x_k}(x) \big) + o(\theta) \notag \\
= \ &\theta \sum_{k=1}^{N_X} D\chi(m_1,x_k)(m_2(k)-m_1(k)) + o(\theta).
\end{align*}
This relation shows that when considering the restriction of $\chi$ to $\mathcal{P}(S)$, it is sufficient to consider the derivative $D\chi(m_1,\cdot)$ by the value taken at the points $x_1$,...,$x_{N_X}$.

\subsection{Algorithms} \label{subsectionAlgo}

We consider the full discretization (in time and space) of the state process. Assume that $\chi$ and $D\chi$ can be computed on $\mathcal{P}(S)$.
We denote by $\mathcal{R}$ the set of reachable probability measures (for the discretized process):
\begin{equation*}
\mathcal{R}= \{ m_{N_T}^{ u} \in \mathcal{P}(S) \,|\, u \in U^{N_T \times N_X} \}
\end{equation*}
We describe now two methods to solve the problem.

\paragraph{Algorithm 1: Gradient method}

The first method that we describe may be seen as a gradient method. It generates a sequence $(m^\ell)_{\ell \geq 1}$ in $\text{conv}(\mathcal{R})$. Assuming that a stopping criterion has been fixed, the method is the following.

\vspace{3mm}

\begin{algorithm}[H]
Choose $u \in U^{N_X \times N_T}$, compute $m^0:= m_{N_T}^{u}$, and set $\ell=0$\;
\While{$m^\ell$ does not satisfy the stopping criterion}{
	\emph{Backward phase}. Compute $D\chi(m^\ell,\cdot)$, and a solution $u^{\ell+1}$ to the problem:
	\begin{equation} \label{eqDiscLinProblem}
	\inf_{u \in U^{N_T \times N_X}} \mathbb{E}\big[D\chi(m^\ell,X_{N_T}^{u}) \big];
	\end{equation}
	\emph{Forward phase}. Compute $\tilde{m}^{\ell+1}:= m_{N_T}^{u^{\ell+1}}$\;
	\emph{Stepsize}. Find an approximate solution $\theta_{\ell+1}$ to:
	\begin{equation} \label{eqStepSize}
	\inf_{\theta \in [0,1]} \chi(\theta m^\ell + (1-\theta) \tilde{m}^{\ell+1}),
	\end{equation}
	and set $m^{\ell+1}= (1-\theta_{\ell+1}) m^\ell + \theta_{\ell+1} \tilde{m}^{\ell+1}$\;
	Set $\ell= \ell+1$\;
}
\KwResult{probability distribution $m^\ell$ at the final time}
\phantom{}
\caption{Gradient method}
\label{algo1}
\end{algorithm}

\vspace{3mm}

Let us describe and comment on the different steps of the method.
\begin{itemize}
\item \emph{Backward phase}. We compute the value function $V:\{ 0,1,...,N_T \} \times \{ 1,...,N_X \}$ which is associated with \eqref{eqDiscLinProblem} by backward induction: $\forall k=1,...,N_X$,
\begin{align}
V_{N_T}(k) & = D\chi(m^\ell,x_k), \quad  \\
V_{j}(k) & = \inf_{u \in U} \, \Big\{ \sum_{k'=1}^{N_X} P_u(k,k') V_{j+1}(k') \Big\},\quad \forall j=N_T-1,...,0. \label{eqDynProgForScheme}
\end{align}
For all $j \in \{ 0,...,N_T-1 \}$ and $k=1,...,N_X$, let $u_j(k)$ be an optimal solution of \eqref{eqDynProgForScheme} --- a solution exists, since $U$ is compact and $P_u(k,k')$ continuous with respect to $u$. We set $u^{\ell+1}= (u_j(k))_{j=0,...,N_T-1,\, k=1,...,N_X}$. Then, $u^{\ell+1}$ is an optimal solution to \eqref{eqDiscLinProblem}.
\item \emph{Forward phase}. The probability measure $\tilde{m}^{\ell+1}= m_{N_T}^{u^{\ell+1}}$ is obtained as follows:
\begin{equation} \label{eqChapman}
\tilde{m}^{\ell+1}= P_{u_{N_T-1}^{\ell}}^{\text{t}}...P_{u_0^{\ell}}^{\text{t}} m_0.
\end{equation}
Eequation \eqref{eqChapman} is nothing but the Chapman-Kolmogorov associated with the Markov chain.
Observe that $\tilde{m}^{\ell+1}$ is an optimal solution to:
$
\inf_{m \in \mathcal{R}} D\chi(m^\ell)m.
$
The probability measure that we have obtained here plays the role of a descent direction.
\item \emph{Stepsize}. In step 3, different approaches can be considered for computing $\theta_{\ell+1}$, depending on properties of $\chi$.
An enumeration technique based on a discretization of $[0,1]$ can be employed. If $\chi$ is convex, then a bisection method can be used. One can also look for a stepsize satisfying the usual stepsize rules  for line-search methods (Armijo, Wolfe-Powell, see \cite[Section 3.4]{Bie10}), rather than looking for an optimal stepsize.
\item \emph{Stopping criterion}.
Different criteria can be considered.
A possible one is the following: given $\varepsilon>0$, we stop at iteration $\ell$ if
\begin{equation}
\varepsilon_\ell:= -\sum_{k=1}^{N_X} m_0(k) V_0(k) + D\chi(m^\ell)m^\ell \leq \varepsilon,
\end{equation}
where $V_0(\cdot)$ is the value function associated with problem \eqref{eqDiscLinProblem}. The value of problem \eqref{eqDiscLinProblem} is then $\sum_{k=1}^{N_X} m_0(k) V_0(k)$. If $\chi$ is convex,
$m^\ell$ is $\varepsilon_{\ell}$-optimal for the discretized problem, similarly to Lemma \ref{lemmaEstimOptimality}.
\end{itemize}

\begin{remark} \begin{enumerate}
\item The sequence of probability measures $(m_\ell)_{\ell \in \mathbb{N}}$  which is generated does not belong to $\mathcal{R}$ in general, but only to its convex envelope. Such an approach is motivated by Lemma \ref{lemmaConvexity}. Note the expression:
\begin{equation*}
m^\ell= \sum_{i=0}^\ell \Big( \Pi_{j=i+1}^{\ell} (1-\theta_j) \Big)\theta_i \tilde{m}^i, \quad \forall \ell, \text{ where: $\theta_0=1$}.
\end{equation*}
\item If $\chi$ is concave (that is to say, if $-\chi$ is convex in the sense of \eqref{eqConvexityChi}), then the optimal solution to \eqref{eqStepSize} is 1. In this case, there is a (discretized) feedback control associated with each probability measure $m^\ell$.
The algorithm falls then into the general framework of \cite{ST11}. In general, if $\chi$ is not convex, we do not expect to find an approximation of a global minimizer.
\end{enumerate}
\end{remark}

\paragraph{Algorithm 2: obtaining a feedback control by a penalization technique}

We suggest here a variant of the method 1 that may enable us to find a solution as a feedback control. The basic idea is the following. Assuming  that at iteration $\ell$, we have a feedback control $u^\ell$ with associated probability distribution $m^\ell= m_{N_T}^{u^\ell}$. Given a coefficient $\alpha>0$, we consider the following linearized and penalized problem:
\begin{equation*} \tag{$P(u^\ell,m^\ell,\alpha)$}  \label{eqPenalization}
\inf_{u \in U^{N_T \times N_X}} \Big\{ \mathbb{E} \big[ D\chi(m^\ell,X_{N_T}^{u}) \big] + \alpha \mathbb{E}\big[ \sum_{j=0}^{N_X} |u_j(X_j^u) - u_j^\ell(X_j^u)|^2 \big] \Big\}.
\end{equation*}
If the coefficient $\alpha$ is large enough, the solution $u'$ to \eqref{eqPenalization} may satisfy $\chi(m_{N_T}^{u'}) < \chi(m_{N_T}^{u^\ell})$. In this way, we avoid the line-search used in the previous algorithm.
Of course, it is not desirable to have a too large coefficient $\alpha$, since this may slow down the procedure. We therefore choose two functions $h^+$ and $h^-:\R_+ \rightarrow \R_+$ satisfying:
\begin{equation*}
0 \leq h^-(\alpha) \leq \alpha \leq h^+(\alpha), \quad \forall \alpha \geq 0,
\end{equation*}
in order to update the coefficient $\alpha$ throughout the procedure, which we present next.

\vspace{2mm}

\begin{algorithm}[H]
Choose $u^0 \in U^{N_X \times N_T}$, compute $m^0= m_{N_T}^{u}$, choose $\alpha_0>0$, set $\ell=0$, and set $q=0$\;
\While{$m^\ell$ does not satisfy the stopping criterion}{
	Solve the linearized and penalized problem $P(u^\ell,m^\ell,\alpha_\ell)$. Denote by $u^{\ell+1}$ its solution, set $m^{\ell+1}= m_{N_T}^{u^{\ell+1}}$ and $q= q+1$\;
	\eIf{$\chi(m^{\ell+1}) < \chi(m^\ell)$}{
   Set $\alpha_{\ell+1}= h^-(\alpha_\ell)$\;
   }{
   \While{$\chi(m^{\ell+1}) \geq \chi(m^\ell)$}{
   		Set $\alpha_\ell= h^+(\alpha_{\ell})$\;
   		Compute a solution $u^{\ell+1}$ to $P(u^\ell,m^\ell,\alpha_\ell)$, set $m^{\ell+1}= m_{N_T}^{u^{\ell+1}}$, and $q= q+1$\;
   		}
   		Set $\alpha_{\ell+1} = \alpha_{\ell}$\;
   }
   Set $\ell= \ell+1$.
}
\KwResult{probability distribution $m^\ell$ at the final time, optimal feedback control $u^\ell$}
\phantom{}
\caption{Variant with a penalization term}
\end{algorithm}

\vspace{2mm}


The variable $q$ does not play any role, it simply counts the number of backward and forward passes.
Note that the problem \ref{eqPenalization} can be solved by dynamic programming, using the following value function: $\forall k= 1,...,N_X$,
\begin{align*}
V_{N_T}(k) & = D\chi(m^\ell,x_k),\\
V_{j}(k) & = \inf_{u \in U} \Big\{ \sum_{k'=1}^{N_X} P_u(k,k') V_{j+1}(k') + \alpha |u-u_j^{\ell}(k)|^2 \Big\}, \ \ \forall j=N_T-1,...,0. \label{eqDynProgForScheme}
\end{align*}
Note also that there are other ways of ``penalizing" the linearized problem, that we do not discuss here.
We finish this section by general remarks for the two algorithms.

\begin{remark} \label{remAlgo}
\begin{enumerate}
\item In general, neither Algorithm 1 nor Algorithm 2 converges to a solution of the discretized problem, since they are only gradient methods.
\item Other discretizations of the SDE are possible, based on implicite finite differences as in \cite{ACCD13}, for example.
\item From a computational point of view, the backward phase is more difficult than the forward phase, since it requires a pointwise minimization. Different techniques for the minimization problem \eqref{eqDynProgForScheme} are suggested in the litterature that we do not discuss here, see e.g. \cite{KKK15}. The simplest method for the minimization is the technique by enumeration on a finite subset of $U$.
\item Denote by $\delta x$ an upper estimate of the diameters of the simpleces of the triangulation. In \cite[Remark 3.3]{CF95}, an error estimate for the approximation of the value function obtained with a semi-Lagrangian scheme involving the ratio $\delta x / \delta t$ (in the case of a Lipschitz data with respect to the state variable) is provided. Therefore, $\delta x$ should be much smaller than $\delta t$. In this situation, the computed probability distribution has an oscillatory behaviour, which is however in our opinion acceptable, see the discussion in section \ref{section5}.
\end{enumerate}
\end{remark}

\section{Examples of cost functions} \label{section4}

In this section, we describe different cost functions $\chi$ and check Assumptions \ref{hypContinuity} and \ref{hypDiff}. The real number $p \geq 2$ is such that $Y_0 \in \mathcal{P}_p(\R^n)$. For the following examples, it is in general only possible to check these assumptions on $\bar{B}_p(R)$. The optimality condition given in Theorem \ref{theoOptiCond} remains true, however, since $\text{cl}(\mathcal{R}(0,Y_0)) \subset \bar{B}_p(R)$, for $R \geq 0$ large enough.

\subsection{Composition of linear costs} \label{section41}

A general class of cost functions can be described as follows. Let $N \in \mathbb{N}$, let $\Psi:\R^N \rightarrow \R$, let $\phi_1$,...,$\phi_N:\R^n \rightarrow \R$ be $N$ continuous functions all dominated by $|x|^p$.
We define then on $\mathcal{P}_p(\R^n)$:
\begin{equation} \label{eqExample1}
\chi(m)= \Psi \Big( \int_{\R^n} \phi_1(x) \dd m(x), ..., \int_{\R^n} \phi_N(x) \dd m(x) \Big).
\end{equation}
For all $R \geq 0$, the continuity on $\bar{B}_p(R)$ is ensured by Lemma \ref{lemmaContinuityDominatedCost}.
Denoting by $y_1$,...,$y_N$ the variables of $\Psi$, the derivative of $\chi$ is given by:
\begin{equation} \label{eqDiffFunctionExp}
D\chi(m)= \sum_{i=1}^N \partial_{y_i} \Psi \Big( \int_{\R^n} \phi_1(x) \dd m(x), ..., \int_{\R^n} \phi_N(x) \dd m(x) \Big) \phi_i(\cdot).
\end{equation}
The differentiability of $\chi$ as well as formula \eqref{eqDiffFunctionExp} can be easily checked.

In this setting, denoting by $H$ the following subset of $\R^N$:
\begin{align*}
\Big\{ y \in \R^N \,|\, \exists m \in \text{cl}(\mathcal{R}(0,Y_0)),\, & \text{such that }\forall i=1,...,N,\, \\
& y_i= \int_{\R^n} \phi_i(x) \dd m_T^{0,Y_0,u}(x) \Big\},
\end{align*}
problem \eqref{eqPb} has the same value as the following problem: $\min_{y \in H} \, \Psi(y)$.
As a consequence of Lemma \ref{lemmaConvexity} and Lemma \ref{lemmaContinuityDominatedCost}, $H$ is compact and convex. Let us denote by $\Phi$ the characteristic function of the set $H$:
\begin{equation*}
\Phi(y)= \begin{cases}
\begin{array}{cl}
0 & \text{ if $y \in H$} \\
+\infty & \text{ otherwise}.
\end{array}
\end{cases}
\end{equation*}
The conjugate function of $\Phi$ is given by:
\begin{align} 
\Phi^*(\lambda)=\ & \sup_{y \in \R^N} \langle \lambda, y \rangle - \Phi(y)
= \sup_{y \in H} \, \langle \lambda, y \rangle \notag \\
=\ & \sup_{u \in \mathcal{U}_0(Y_0)} \Big\{ \int_{\R^n} \sum_{i=1}^N \lambda_i \phi_i(x) \dd m_T^{0,Y_0,u}(x) \Big\}. \label{eqConjPhi}
\end{align}
By Fenchel-Moreau-Rockafellar theorem,
\begin{equation*}
\Phi(y) = \sup_{\lambda \in \R^N} \big\{ \langle \lambda, y \rangle - \Phi^*(\lambda) \big\},
\end{equation*}
or equivalently, $H$ can be described as an intersection of hyperplanes:
\begin{equation*}
H= \cup_{\lambda \in \R^N} \{ y \,|\, \langle \lambda, y \rangle \leq \Phi^*(\lambda) \}.
\end{equation*}
This means that one can build an outer polyhedral approximation of $H$, by selecting different values of $\lambda \in \R^n$, and computing the value of $\Phi^*$ by solving the corresponding standard optimal control problem \eqref{eqConjPhi}.

\begin{lemma} \label{lemmaAssSatisfied}
If the functions $\phi_1$,...,$\phi_N$ are Lipschitz continuous and dominated by $|x|^p$ and if $\Psi$ has a Lipschitz derivative, then Assumption \ref{hypLipDerivative} is satisfied.
\end{lemma}

\begin{proof}
Let $K_\phi$ be the Lipschitz modulus of the functions $\phi_1$,...,$\phi_N$, let $K_{D\Psi}$ be the Lipschitz modulus of $D\Psi$. Let $m_1$, $m_2$, $m_3$, and $m_4$ in $\text{cl}(\mathcal{R}(0,Y_0))$. We set:
\begin{align*}
& y_1= \Big( \int_{\R^n} \phi_1 \dd m_1,...,\int_{\R^n} \phi_N \dd m_1 \Big), \\
& y_2= \Big( \int_{\R^n} \phi_1 \dd m_2,...,\int_{\R^n} \phi_N \dd m_2 \Big).
\end{align*}
One can show that:
\begin{align*}
\big| \big( D\chi(m_2)-D\chi(m_1) \big)(m_4-m_3) \big|
\leq \ & |D\Psi(y_2)-D\Psi(y_1)| K_\phi d_1(m_3,m_4) \\
\leq \ & K_{D\Psi} K_\phi^2 \, d_1(m_1,m_2) d_1(m_3,m_4).
\end{align*}
The estimate \eqref{hypLipDerivative1} follows by taking $m_3=m_1$ and $m_4= m_2$. The estimate \eqref{hypLipDerivative2} follows from: $d_1(m_3,m_4) \leq D$, where $D$ is defined by \eqref{eqDiameter}.
\end{proof}

We finish this paragraph by recalling a density property, stated and proved in \cite[Section 5.3]{Car12}.
For all $R \geq 0$, any function $\chi: \bar{B}_p(R) \rightarrow \R$ that can be written in the form \eqref{eqExample1}, with continuous functions $\phi_1$,...,$\phi_N$, all dominated by $|x|^p$, is called polynomial function on $\bar{B}_p(R)$.

\begin{proposition}
For all $R \geq 0$, the set of polynomial functions on $\bar{B}_p(R)$ is dense in the set of continuous functions (for the $d_1$-distance) on $\bar{B}_p(R)$, that is to say, for all continuous cost function $\chi:\bar{B}_p(R) \rightarrow \R$, for all $\varepsilon>0$, there exists a polynomial function $\tilde{\chi}$ such that $\sup_{m \in \bar{B}_p(R)} |\chi(m)-\tilde{\chi}(m)| \leq \varepsilon$.
\end{proposition}

\begin{proof}
Since for all $R \geq 0$, $\bar{B}_p(R)$ is compact for the $d_1$-distance, this result is a direct consequence of the Stone-Weierstrass theorem, see \cite[Section 5.3]{Car12} for details. Note that it is possible to restrict the set of polynomial functions to polynomial functions involving functions $\phi_1$,...,$\phi_N$ which are infinitely many times differentiable, with a compact support.
\end{proof}

For measures in $\R$ and for $r \leq p$, the central moment of order $r$, denoted by $\mu_r$ is a polynomial function of $\mathcal{P}_p$. Indeed,
\begin{align*}
\mu_r(m)=\ & \int_{\R} \Big( x - \int_{\R} y \dd m(y) \Big)^r \dd m(x) \\
=\ & \sum_{i=0}^r \binom{r}{i} \Big(\int_{\R} x^i \dd m(x) \Big) \Big( - \int_{\R} x \dd m(x) \Big)^{r-i}.
\end{align*}
Finally, to obtain a polynomial function, it suffices to set $\phi_i= x^i$ for $i=1,...,r$ and $\Psi(y_1,...,y_r)= (-y_1)^r + \sum_{i=1}^r \binom{r}{i} y_i (-y_1)^{r-i}$.
In particular, for $r = 2$ and $p \geq 2$, the variance is given by taking $\Psi(y_1,y_2)= y_2 - y_1^2$, which is a concave function.

\subsection{Wasserstein distance}

In various domains (in quantum mecanics for example, see \cite[Section 4]{HSBTZ13}, \cite{AB13}), one tries to reach a given probability measure $m^*$. Different cost functions can be employed to measure the distance from a given probability measure $m$ to the prescribed one, $m^*$. Note that the usual distances, such as the $L^2$-norm, requires that $m$ has a density function and are in general not continuous for the $d_1$-distance. The Wasserstein distance defined below has good continuity and differentiability properties.

For $c:\R^n \times \R^n$, consider the optimal transport problem
\begin{equation*}
\chi(m) = \inf_{\pi \in \Pi(m,m^*)} \int_{\R^n \times \R^n} c(x,y) \dd \pi(x,y)
\end{equation*}
and its dual:
\begin{align*}
 & \sup_{\phi \in L^1(m),\, \psi \in L^1(m^*)} \Big\{ \int_{\R^n} \phi(x) \dd m(x) + \int_{\R^n} \psi(x) \dd m^*(x) \Big\} \\
& \text{such that: } \phi(x)+\psi(y) \leq c(x,y),\, \forall (x,y) \in \R^n \times \R^n.
\end{align*}

Let $1 \leq q < p$ and consider the case of the Wasserstein distance $c(x,y)= |y-x|^q$, so that $\chi(m)= d_q^q(m,m^*)$.
Let $(m_k)_{k \in \mathbb{N}}$ be a converging sequence in $\bar{B}_p(R)$ for the $d_1$-distance with limit $\bar{m}$. By theorem \cite[Definition 6.8/Theorem 6.9]{Vil09}, the sequence is weakly converging in $\mathcal{P}(\R^n)$ and by Lemma \ref{lemmaContinuityDominatedCost}, $\int_{\R^n} |x|^q \dd m_k(x) \rightarrow \int_{\R^n} |x|^q \dd m(x)$. Therefore, applying once again \cite[Definition 6.8/Theorem 6.9]{Vil09}, we obtain the convergence for the $d_p$-distance and therefore the continuity of $\chi$.

For the choice $c(x,y)= |y-x|^q$, the conditions of \cite[Theorem 5.10]{Vil09} are satisfied and therefore, the primal and dual problem have the same value and have both an optimal solution. In the dual formulation, $\chi$ is expressed as the supremum of affine cost functions, thus is convex (in the sense of \eqref{eqConvexityChi}). We also have a sub-differentiability property on $\mathcal{P}_q(\R^n)$, given by equation \eqref{eqSubDiffWasserstein}. Let $m_1$ and $m_2 \in \mathcal{P}_q(\R^n)$ and let $(\phi_1,\psi_1)$ be a solution to the dual problem associated with $m_1$. Since for all $x \in \R^n$,
$
\phi_1(x) \leq |x|^q - \psi(0),
$
we have that $\int_{\R^n} \phi_1 \dd m_2 \in [-\infty,+\infty)$ and moreover, $\phi_1 \in L^1(\mu_2)$ if and only if $\int_{\R^n} \phi_1 \dd m_2 > - \infty$. Therefore, we obtain:
\begin{equation} \label{eqSubDiffWasserstein}
\chi(m_2) \geq \int_{\R^n} \phi_1 \dd m_2 + \int_{\R^n} \psi_1 \dd m^* = \chi(m_1) + \int_{\R^n} \phi_1 \dd( m_2- m_1).
\end{equation}
Indeed, if $\int_{\R^n} \phi_1 \dd m_2 = -\infty$, the inequality is trivial, otherwise, $\phi_1 \in L^1(m_2)$, thus $\phi_1$ is sub-optimal in the dual problem (associated with $m_2$).

\subsection{Conditional value at risk}

A popular cost functional in stochastic optimization is the Conditional Value at Risk (CVaR). Note that the CVaR belongs to the important class of coherent risk measures \cite{ADEH99}. The CVaR of a random variable in $\R$ (typically modelling losses) can be easily understood when the distribution has a density. Given a probability level $\beta \in (0,1)$, one has to define first the value at risk (VaR) as the smallest value $\alpha \in \R$ such that with a probability greater to $\beta$, the losses do not exceed $\alpha$. If the distribution has a density, then
the CVaR is the conditional expectation of the losses, under the condition that they exceed $\alpha$.

In this subsection, we describe the CVaR of a random variable as the value of an optimal transportation problem involving its probability distribution. This enables to derive easily a well-known formula \cite[Theorem 1]{RU00}, to prove the concavity and the continuity of the CVaR, and to prove a super-differentiability property.

Let $m \in \mathcal{P}(\R)$ be a probability measure on $\R$ and let $\Omega$ be a set with exactly two elements that we arbitrarily denote $\omega_0$ and $\omega_1$. Let $\beta \in (0,1)$, let $\nu \in \mathcal{P}(\Omega)$ be the probability measure on $\Omega$ defined by: $\nu(\{\omega_0\})= \beta$, $\nu(\{ \omega_1 \})= 1-\beta$.
We define the set of transportation plans $\Pi(m,\nu)$ between $\mu$ and $\nu$ as:
\begin{align*}
\Pi(m,\nu)= \Big\{ \pi \in \mathcal{P}(\R \times \Omega) \,|\, & \pi(A \times \Omega)= m(A),\ \forall A \in \sigma(\R), \\
& \pi(\R \times \omega_0)= \beta,\quad \pi(\R \times \omega_1)= 1-\beta \Big\},
\end{align*}
where $\sigma(\R)$ is the $\sigma$-algebra of Borel subsets of $\R$.
We also set:
\begin{equation*}
c:(x,\omega) \in \R \times \Omega \mapsto \begin{cases}
\begin{array}{cl}
0 & \text{if }\omega= \omega_0 \\
x & \text{if }\omega= \omega_1.
\end{array}
\end{cases}
\end{equation*}
The CVaR (with level $\beta$) of the probability measure $m$ is now defined by:
\begin{equation} \label{eqPrimaCVaR}
\text{CVaR}(m)= \frac{1}{1-\beta} \sup_{\pi \in \Pi(m,\nu)} \Big\{ \int_{\R \times \Omega} c(x,\omega) \dd \pi(x,\omega) \Big\}.
\end{equation}
The dual problem is therefore:
\begin{align*}
& \inf_{\begin{subarray}{c} \phi \in L^1(m) \\ \psi \in L^1(\nu) \end{subarray} }
\Big\{ \int_{\R} \phi(x) \dd m(x) + \int_{\Omega} \psi(\omega) \dd \nu(\omega) \Big\}, \\
& \qquad \text{s.t. } \phi(x) + \psi(\omega) \geq c(x,\omega), \forall (x,\omega) \in \R \times \Omega.
\end{align*}
Let us analyse the set of feasible functions $(\phi,\psi)$. It is sufficient to describe $\psi$ as a vector of $\R^2$. We therefore write $\psi=(\psi_0,\psi_1)$, where $\psi_0= \psi(\omega_0)$ and $\psi_1= \psi(\omega_1)$. For a given function $\psi \in \R^2$, it is sufficient to consider the function $\phi$ defined by:
\begin{equation*}
\phi(x)= \max_{i=0,1} \big\{ c(x,\omega_i) - \psi_i \big\} =  \max (x-\psi_1,-\psi_0)
\end{equation*}
in the dual problem.
We recover then the formula of \cite{RU00}:
\begin{align}
  & (1-\beta) \text{CVaR}(m) \notag \\
= & \inf_{(\psi_0,\psi_1) \in \R^2} \Big\{ \int_{\R} \max(x-\psi_1,-\psi_0) \dd m(x) + \beta \psi_0 + (1-\beta) \psi_1 \Big\} \notag \\
= & \inf_{(\psi_0,\psi_1) \in \R^2} \Big\{ \int_{\R} \max(x-\psi_1+\psi_0,0) \dd m(x) + (1-\beta)(\psi_1-\psi_0) \Big\} \notag \\
= & \ \inf_{\psi \in \R} \Big\{ \int_{\R} (x-\psi)_+ \dd m(x) + (1-\beta) \psi \Big\}. \label{eqDualCVaR}
\end{align}
Let us define:
\begin{equation*}
A(m)= \Big\{ \alpha \in \R \,|\, m \big( (-\infty,\alpha] \big) \geq \beta, \, m\big( [\alpha,+\infty) \big) \geq 1-\beta \Big\},
\end{equation*}
The set $A(m)$ is a closed and bounded interval. Let us also define:
\begin{equation*}
\alpha_-(m) = \min A(m) \quad \text{and} \quad \alpha_+(m) = \max A(m).
\end{equation*}
For a given $m \in \mathcal{P}(\R)$, it is easy to show the existence and uniqueness of non-negative measures $m_-$ and $m_+$ on $\R$, having a support included respectively in $(-\infty,\alpha_-]$ and $[\alpha_+,+\infty)$, and such that $m\big( (-\infty,\alpha_-] \big)= \beta$ and $m\big( [\alpha_+,+\infty) \big)= 1-\beta$.

\begin{lemma}
The unique solution to the primal problem \eqref{eqPrimaCVaR} is given by:
\begin{equation*}
\pi= (m_- \times \delta_{\omega_0}) + (m_+ \times \delta_{\omega_1}).
\end{equation*}
The interval $A(m)$ is the set of optimal solutions to the dual problem \eqref{eqDualCVaR}.
\end{lemma}

\begin{proof}
We only prove that $A(m)$ is the set of optimal solutions of the dual problem.
We denote by $D(\alpha)$ the dual criterion:
\begin{equation*}
D(\alpha) = \int_{\R} (x-\alpha)_+ \dd m(x) + (1-\beta) \alpha.
\end{equation*}
For $\alpha \geq \alpha_-$, it holds:
\begin{align*}
D(\alpha)=\ & \int_{\R} (x-\alpha)_+ \dd m_+(x) + (1-\beta) \alpha \\
=\ & \int_{\R} \big( (x-\alpha)_+  + (\alpha-\alpha_-) \big) \dd m_+(x) + (1-\beta) \alpha_- \\
=\ & \int_{\R} (x-\alpha_-) \dd m_+(x) + (1-\beta) \alpha_- + \int_{\R} (\alpha-x)_+ \dd m_+(x) \\
=\ & D(\alpha_-) + \int_{\R} (\alpha-x)_+ \dd m^+(x).
\end{align*}
We let the reader check that if $\alpha \in [\alpha_-,\alpha_+]$, then $\int_{\R} (\alpha-x)_+ \dd m^+(x)= 0$, and if $\alpha > \alpha_+$, then $\int_{\R} (\alpha-x)_+ \dd m^+(x)> 0$.

Now, let $\alpha \leq \alpha_-$. Using $x-\alpha= (x-\alpha)_+ - (\alpha-x)_+$, we obtain:
\begin{align*}
D(\alpha)=\ & \int_{\R} (x-\alpha) \dd m_+(x) + (1-\beta) \alpha_- + \int_{\R} (x-\alpha)_+ \dd m^-(x) \\
=\ & \int_{\R} (x-\alpha_-) \dd m_+(x) + (1-\beta) \alpha_- + \int_{\R} (x-\alpha)_+ \dd m^-(x) \\
=\ & D(\alpha_-) + \int_{\R} (x-\alpha)_+ \dd m^-(x).
\end{align*}
Finally, the reader can check that for $\alpha< \alpha_-$, $\int_{\R} (x-\alpha)_+ \dd m^-(x) < 0$.
\end{proof}

In the dual formulation, the CVaR is expressed as the infimum of affine functions. Therefore, it is concave. Observe that for all $\psi \in \R$, $x \mapsto (x-\psi)_+$ is 1 Lipschitz, therefore, for all $m_1$ and $m_2 \in \mathcal{P}_1(\R)$,
\begin{align*}
(1-\beta) \text{CVaR}(m_2) \leq \ &
\int_{\R} (x-\psi)_+ \dd m_2(x) + (1-\beta) \psi \\
\leq \ & \int_{\R} (x-\psi)_+ \dd m_1(x) + (1-\beta) \psi + d_1(m_1,m_2).
\end{align*}
Minimizing with respect to $\psi$, we obtain that:
\begin{equation*}
\text{CVaR}(m_2)-\text{CVaR}(m_1) \leq \frac{1}{1-\beta} \, d_1(m_1,m_2).
\end{equation*}
Exchanging $m_1$ and $m_2$ in the previous inequality, we obtain the Lipschitz-continuity of the CVaR. Similarly to \eqref{eqSubDiffWasserstein}, $\chi$ is super-differentiable in the following sense:
\begin{equation}
\chi(m_2) \leq \chi(m_1) + \frac{1}{1-\beta} \int_{\R} (x-\psi)_+ \dd (m_2(x)-m_1(x)),
\end{equation}
for all $\psi \in A(m_1)$.

\subsection{Integral of interactions}

Given a function $\phi:\R^{2n} \rightarrow \R$, we can define:
\begin{equation*}
\chi(m)= \int_{\R^{2n}} \phi(x,y) \dd m(x) \dd m(y).
\end{equation*}
A possible choice of $\phi$ is the following: $\phi(x,y)= \varphi(|y-x|)$ : if $\varphi$ is increasing, the diffusion of the state variable is penalised. One can for example easily check that:
\begin{equation*}
\int_{\R^{2n}} \frac{1}{2} | y-x |^2 \dd m(x) \dd m(y) =  \int_{\R^n} \Big( y- \int_{\R^n} x \dd m(x) \Big)^2 \dd m(y)= \text{Var}(m).
\end{equation*}

Assume that for all $\varepsilon>0$, there exists $r>0$ such that if $|x| \geq r$ or $|y| \geq r$, then $|\phi(x,y)| \leq \varepsilon \max(|x|^p,|y|^p)$. Let $(m_k)_{k \in \mathbb{N}}$ be a converging sequence in $\bar{B}_p(R)$ for the  $d_1$-distance with limit $m$. We let the reader check that the sequence $(m_k \times m_k)_k$ converges to $(m \times m)$ for the $d_1$-distance of $\R^{2n}$. Applying Lemma \ref{lemmaContinuityDominatedCost} (in $\R^{2n}$), the continuity of $\chi$ follows.
The derivative is given by:
\begin{equation}
D\chi(m,x)= \int_{\R^m} \big( \phi(x,y) + \phi(y,x) \big) \dd m(y).
\end{equation}
Assumption \ref{hypDiff} holds, as a consequence of the following identity: for all $m_1$ and $m_2 \in \mathcal{P}_p(\R^n)$, for all $\theta \in [0,1]$,
\begin{align*}
& \chi((1-\theta)m_1 + \theta m_2) \notag \\
=\ & \chi(m_1) + \theta \int_{\R^{2n}} \big(\phi(x,y) + \phi(y,x)\big) \dd m_1(x)(\text{d} m_2(y)-\dd m_1(y)) \\
& \qquad + \theta^2 \int_{\R^{2n}} \phi(x,y) (\text{d} m_2(x)- \dd m_1(x))(\text{d} m_2(y)-\dd m_1(y)).
\end{align*}

\section{Numerical results} \label{section5}

\begin{figure}[p]
\label{pageTestCase1}
\fbox{
\begin{minipage}[c]{0.95\linewidth}

\vspace{3mm}

\centering{\textbf{Test Case 1}} \\
\vspace{5mm}
\begin{tabular}{|c||cc|ccc|} \hline
                 & \multicolumn{2}{c|}{\text{Algorithm 1}} & \multicolumn{3}{c|}{\text{Algorithm 2}} \\ \hline
$\ell$ & Cost: $\chi(m^\ell)$ & Criterion: $\varepsilon_{\ell}$ & Cost: $\chi(m^\ell)$ & Criterion: $\varepsilon_{\ell}$ & \phantom{$\big|^|$}$q$\phantom{$\big|^|$} \\ \hline
0  & 0.8742 & 0.4249 & 0.8742 & 0.4249 & 0 \\
5  & 0.5558 & 0.0492 & 0.8116 & 0.3441 & 5 \\
10 & 0.5510 & 0.0433 & 0.5335 & 0.0137 & 10 \\
15 & 0.5461 & 0.0554 & 0.5330 & 0.0062 & 19 \\
20 & 0.5367 & 0.0296 & 0.5220 & 0.0102 & 31 \\
25 & 0.5333 & 0.0158 & 0.5206 & 0.0012 & 42 \\
30 & 0.5313 & 0.0202 & 0.5204 & 0.0003 & 53 \\
35 & 0.5300 & 0.0107 & 0.5203 & 0.0003 & 62 \\
40 & 0.5284 & 0.0320 & 0.5203 & 0.0003 & 74 \\
45 & 0.5274 & 0.0082 & 0.5203 & 0.0003 & 83 \\
50 & 0.5267 & 0.0252 & 0.5203 & 0.0003 & 94 \\
55 & 0.5241 & 0.0097 & 0.5203 & 0.0011 & 103 \\
60 & 0.5205 & 0.0026 & 0.5203 & 0.0011 & 113 \\
65 & 0.5204 & 0.0018 & 0.5203 & 0.0011 & 123 \\
70 & 0.5204 & 0.0018 & 0.5203 & 0.0011 & 133 \\
\hline
\end{tabular}
\caption{Convergence results}
\label{figCaseTransportConvergence}

   \begin{minipage}[c]{.47\linewidth}
      \includegraphics[width=\textwidth]{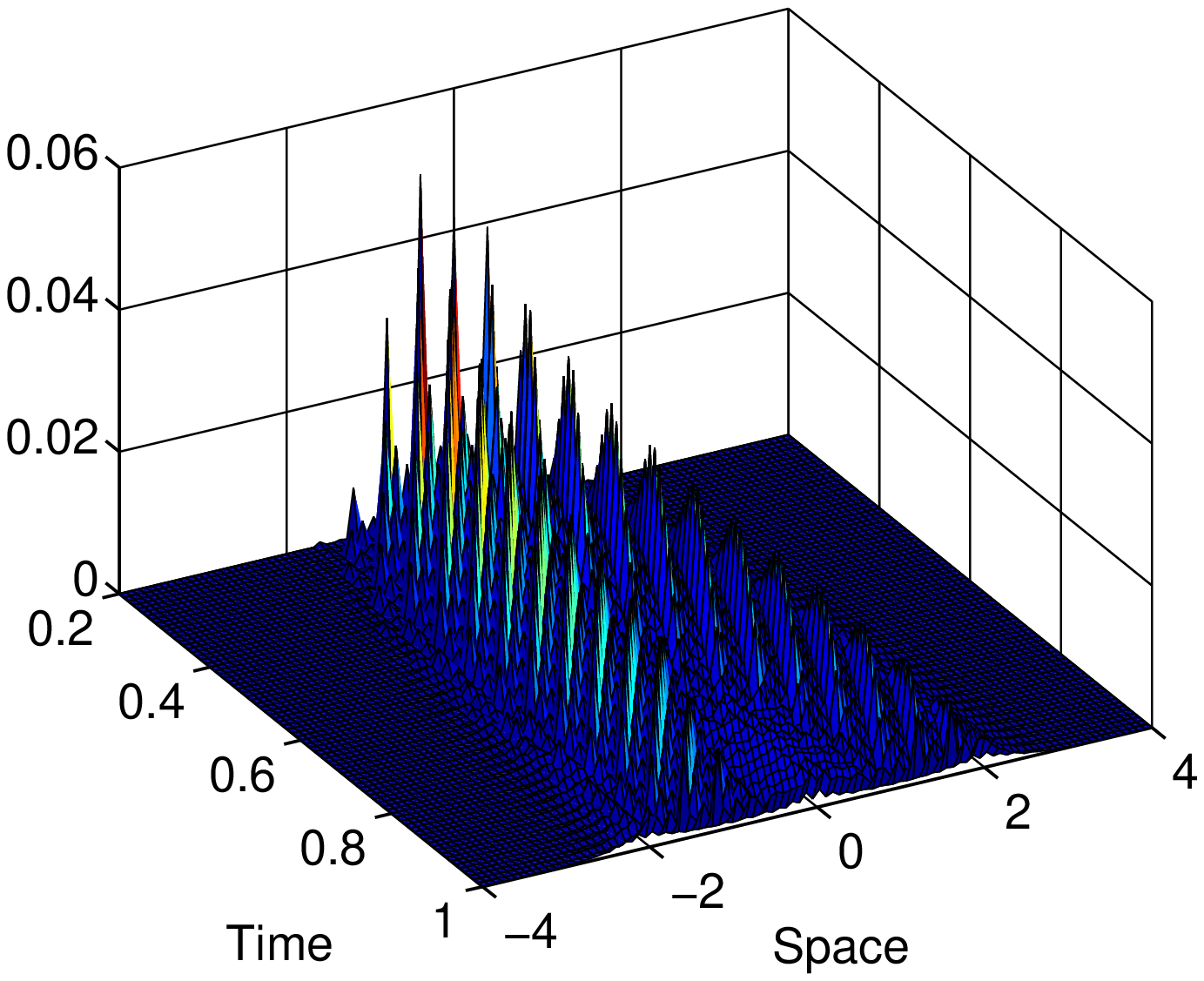} 
      \caption{\newline Probability distribution}
      \label{figTransport1}
      \vspace{2mm}
      \includegraphics[width=\textwidth]{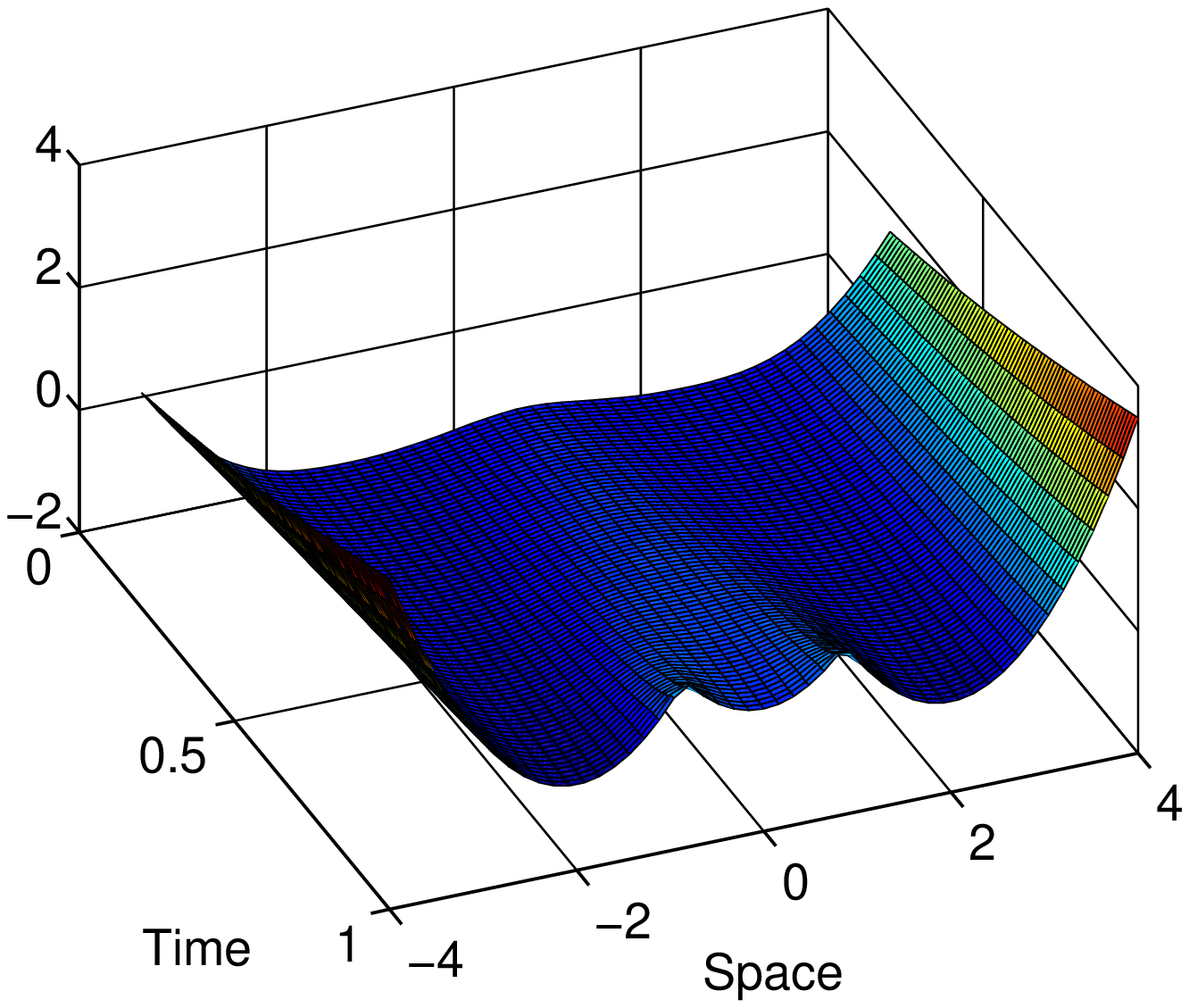} 
      \caption{\newline Value function}
      \label{figTransport3}
   \end{minipage} \hfill
   \begin{minipage}[c]{.47\linewidth}
   \includegraphics[width=\textwidth]{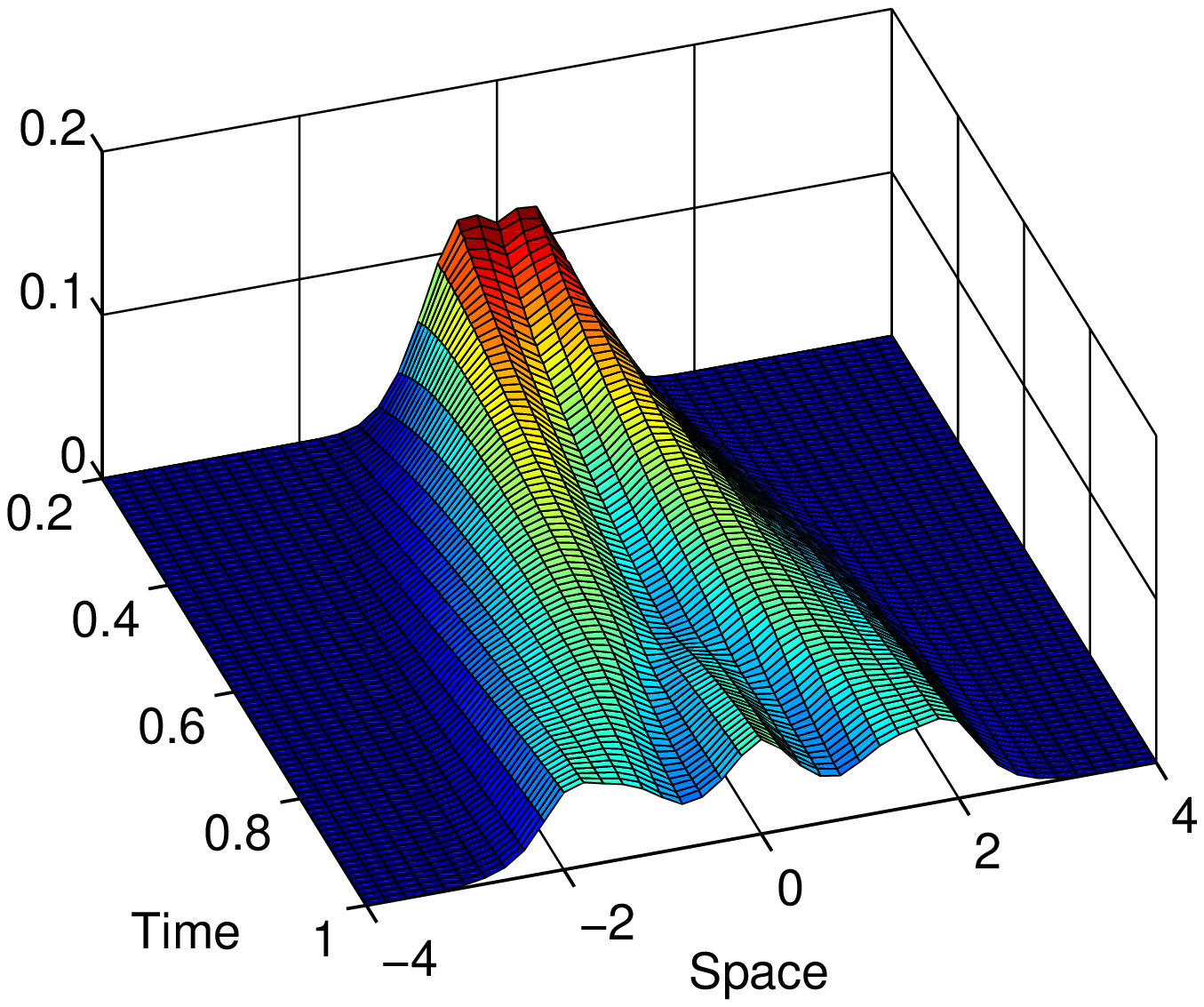}
      \caption{\newline Regularized  distribution}
      \label{figTransport2}
      \vspace{2mm}
      \includegraphics[width=\textwidth]{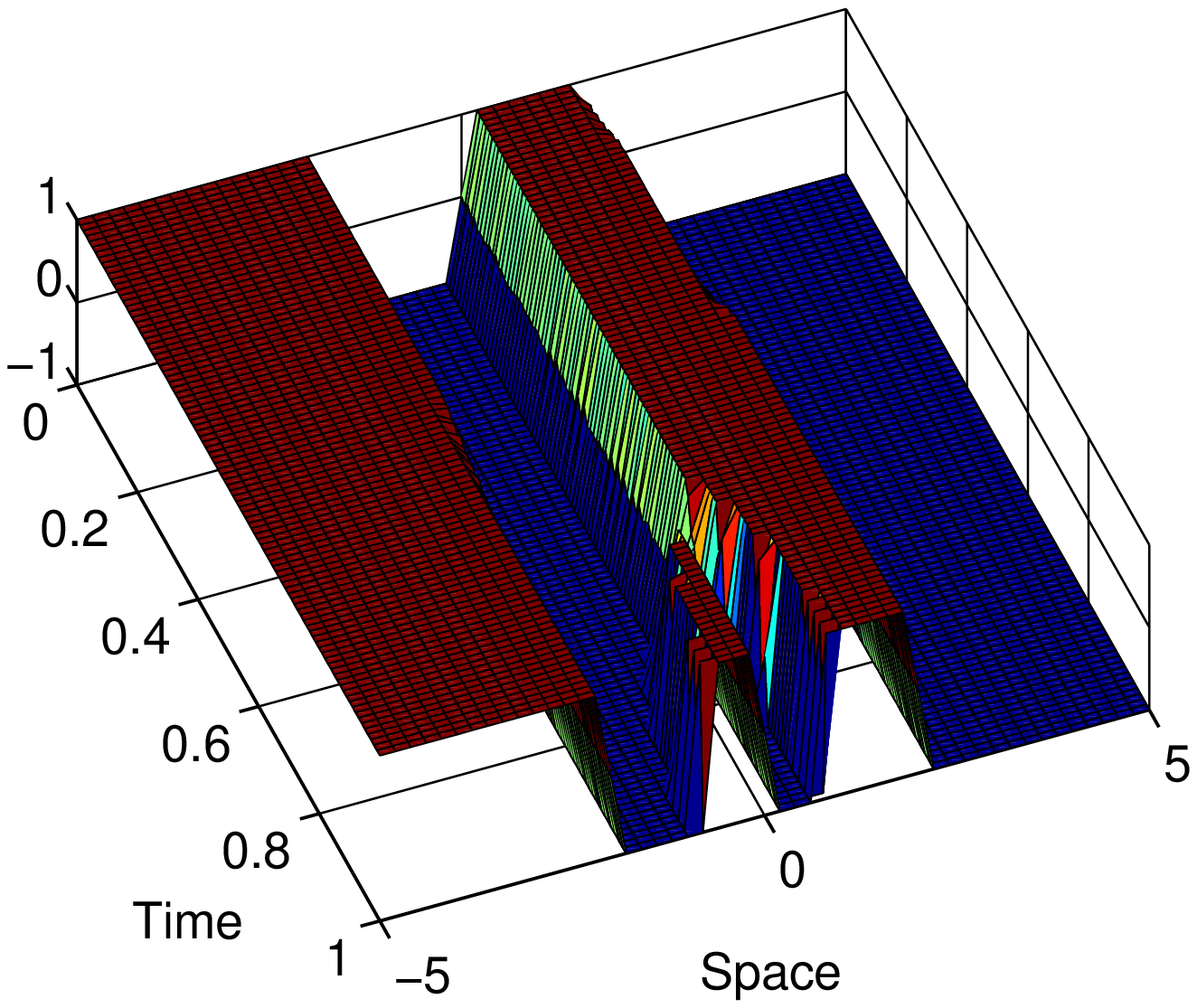}
      \caption{\newline Optimal control}
      \label{figTransport4}
   \end{minipage}
   \vspace{3mm} 
\end{minipage}
}
\end{figure}

\begin{figure}[p]
\fbox{
   \begin{minipage}[c]{.45\linewidth}
   \vspace{5mm}   	
   	\centering{\textbf{Test Case 2}}
      \includegraphics[width=\textwidth]{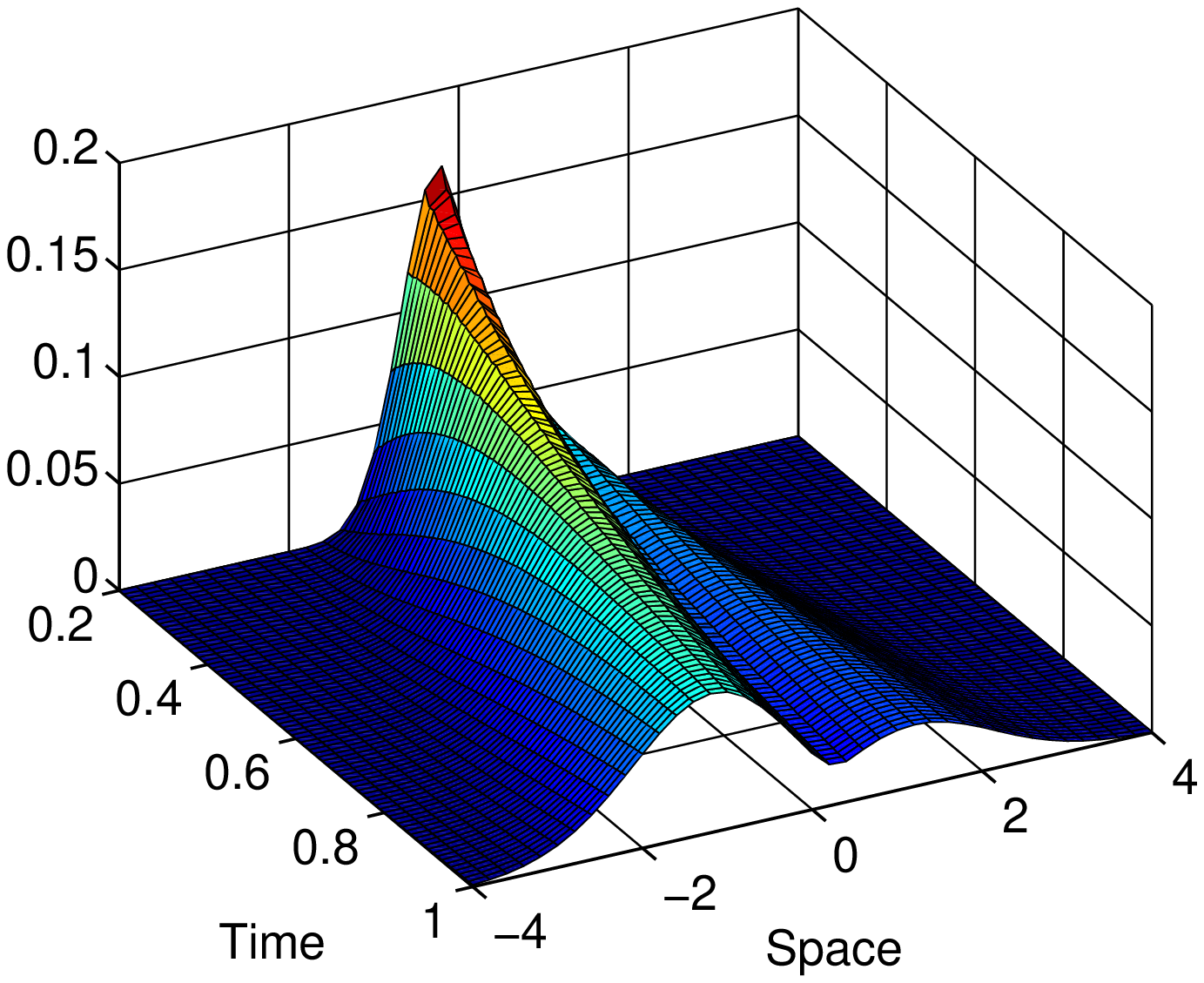} 
      \caption{\newline Regularized distribution}
      \label{figMixConvexe1}
      \vspace{2mm}
      \includegraphics[width=\textwidth]{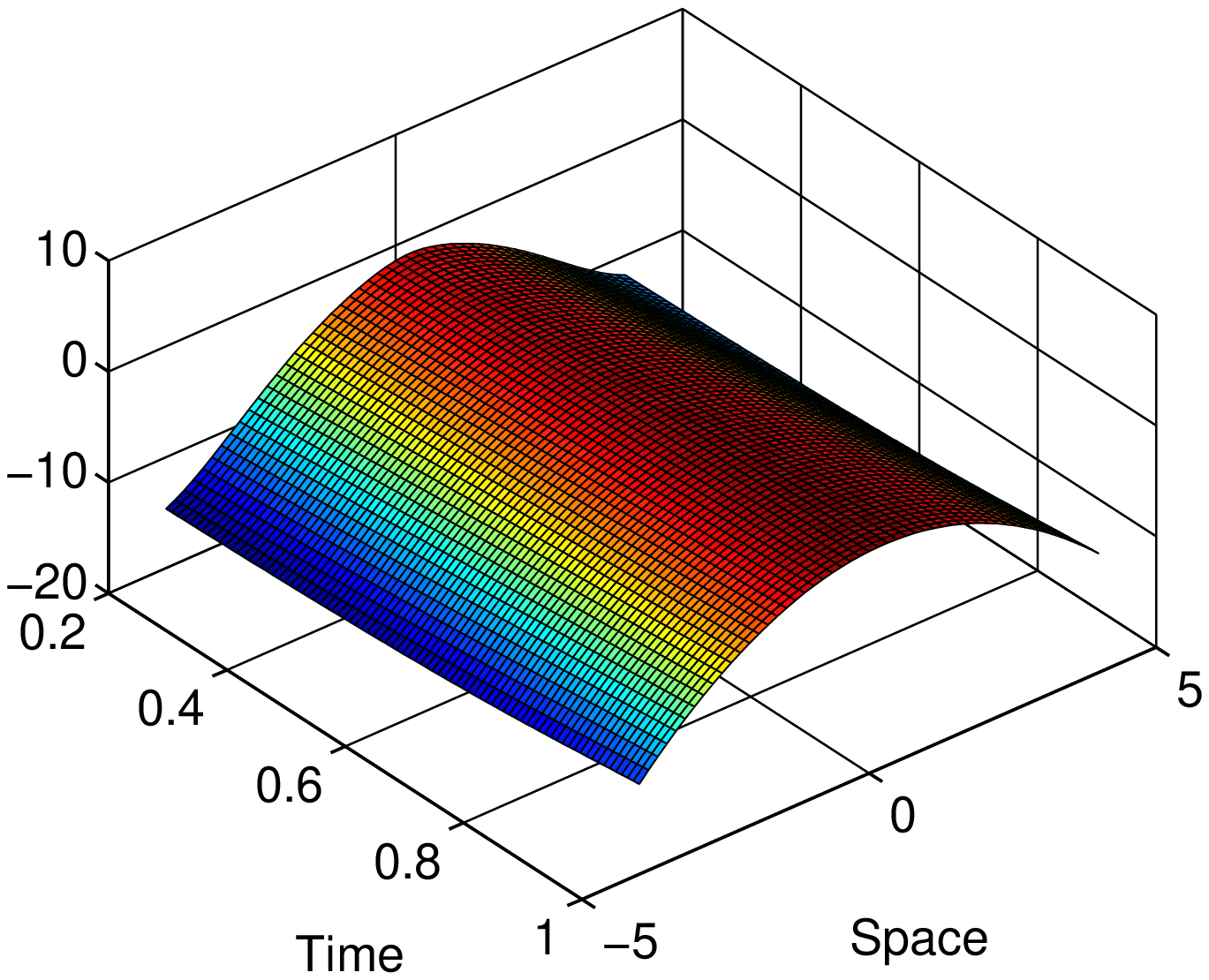}
      \caption{Value function}
      \label{figMixConvexe2}
      \vspace{2mm}
      \includegraphics[width=\textwidth]{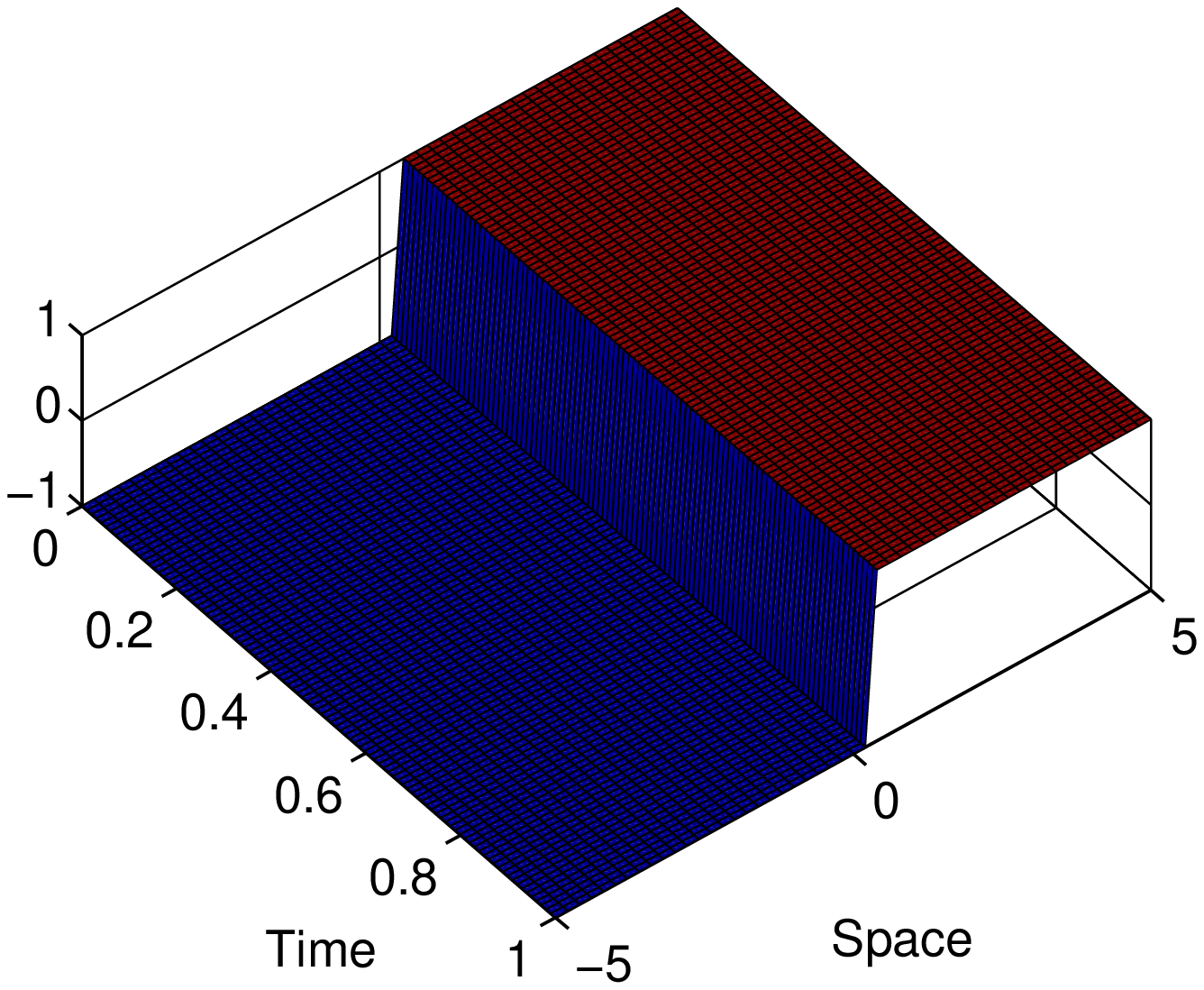}
      \caption{\newline Optimal Control}
      \label{figMixConvexe3}
      \vspace{4mm}
   \end{minipage} } \hfill
\fbox{
   \begin{minipage}[c]{.45\linewidth}
   \vspace{5mm}   
   \centering{\textbf{Test Case 3}}
      \includegraphics[width=\textwidth]{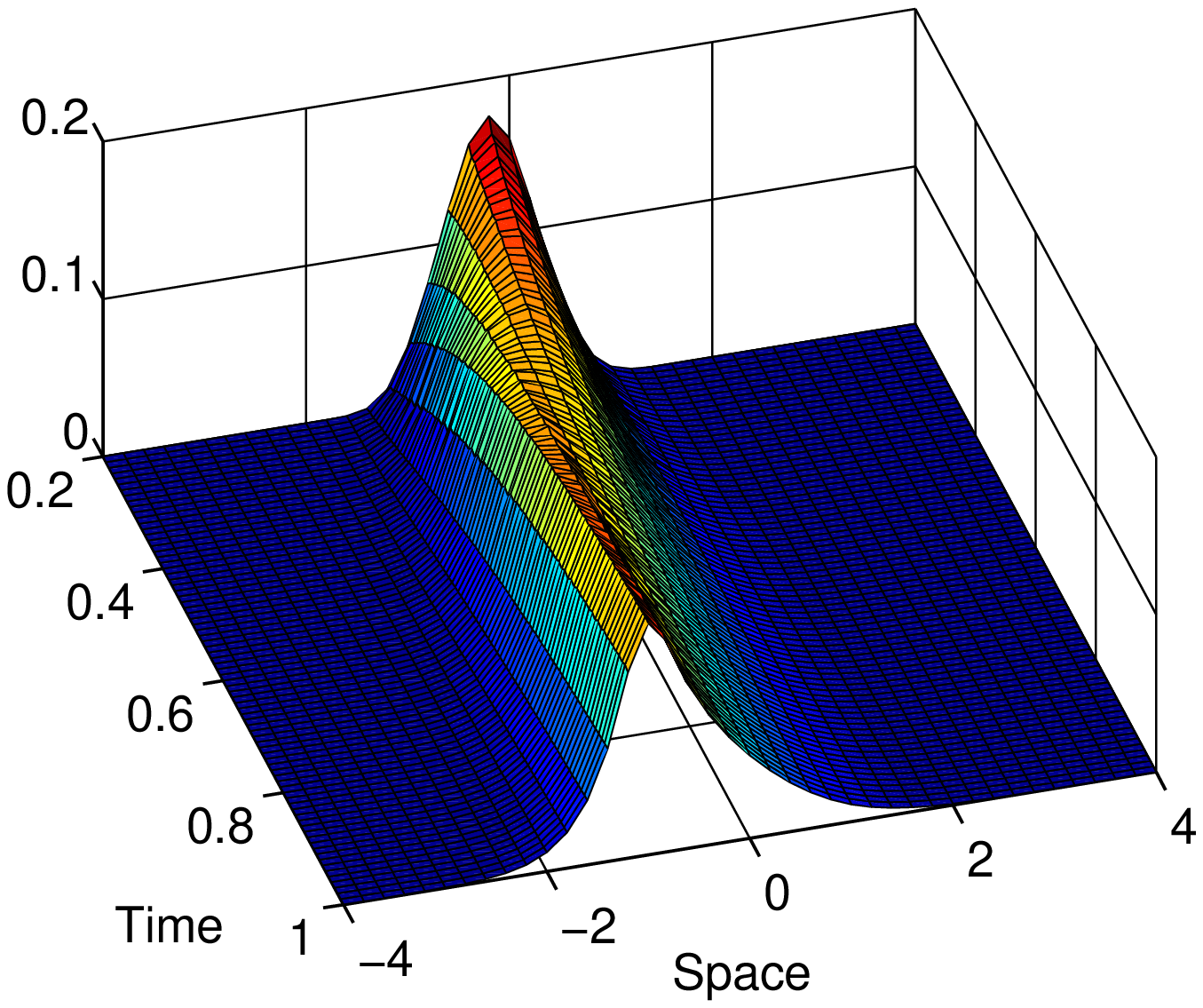} 
      \caption{\newline Regularized distribution}
      \label{figMixConcave1}
      \vspace{2mm}
      \includegraphics[width=\textwidth]{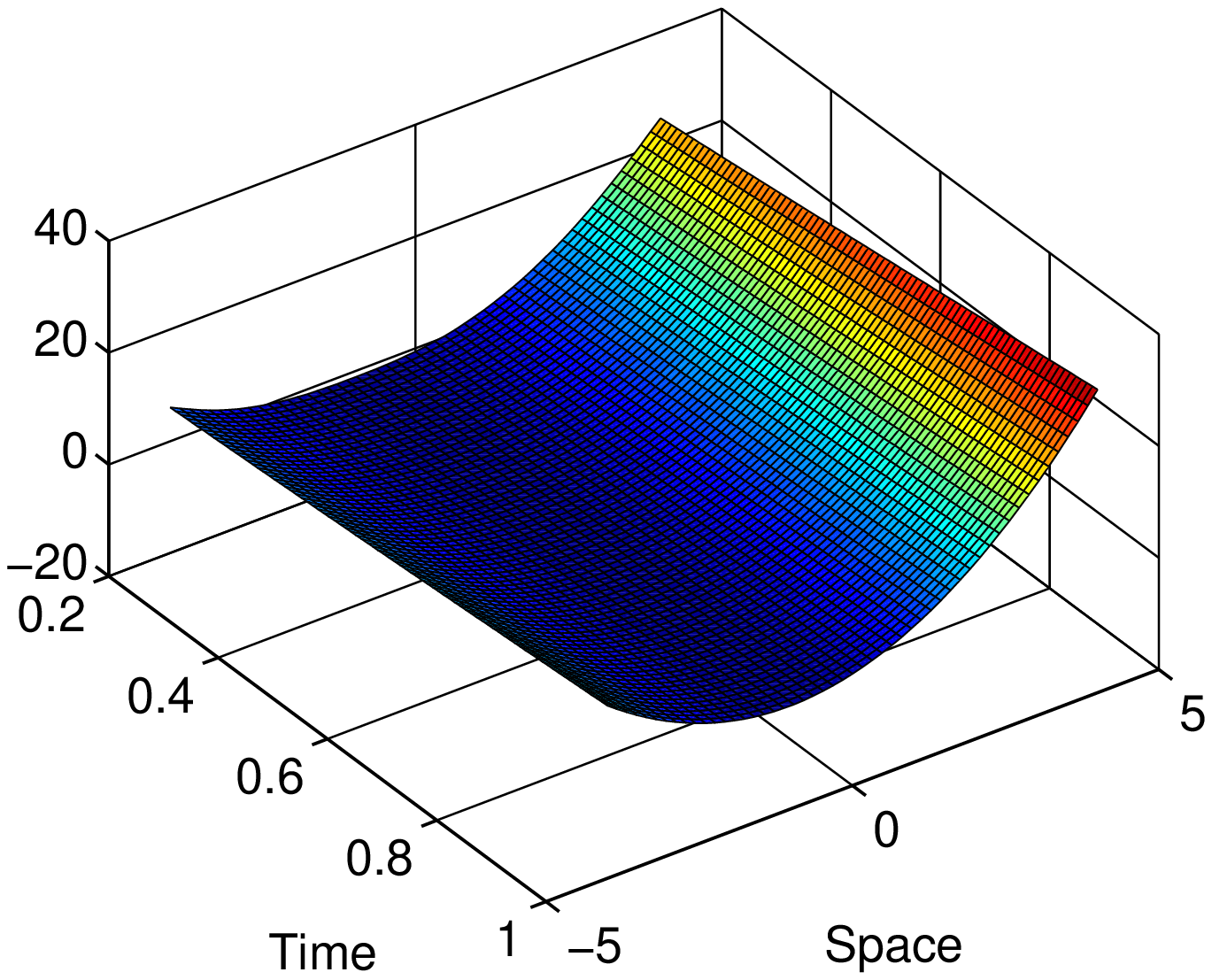}
      \caption{Value function}
      \label{figMixConcave2}
      \vspace{2mm}
      \includegraphics[width=\textwidth]{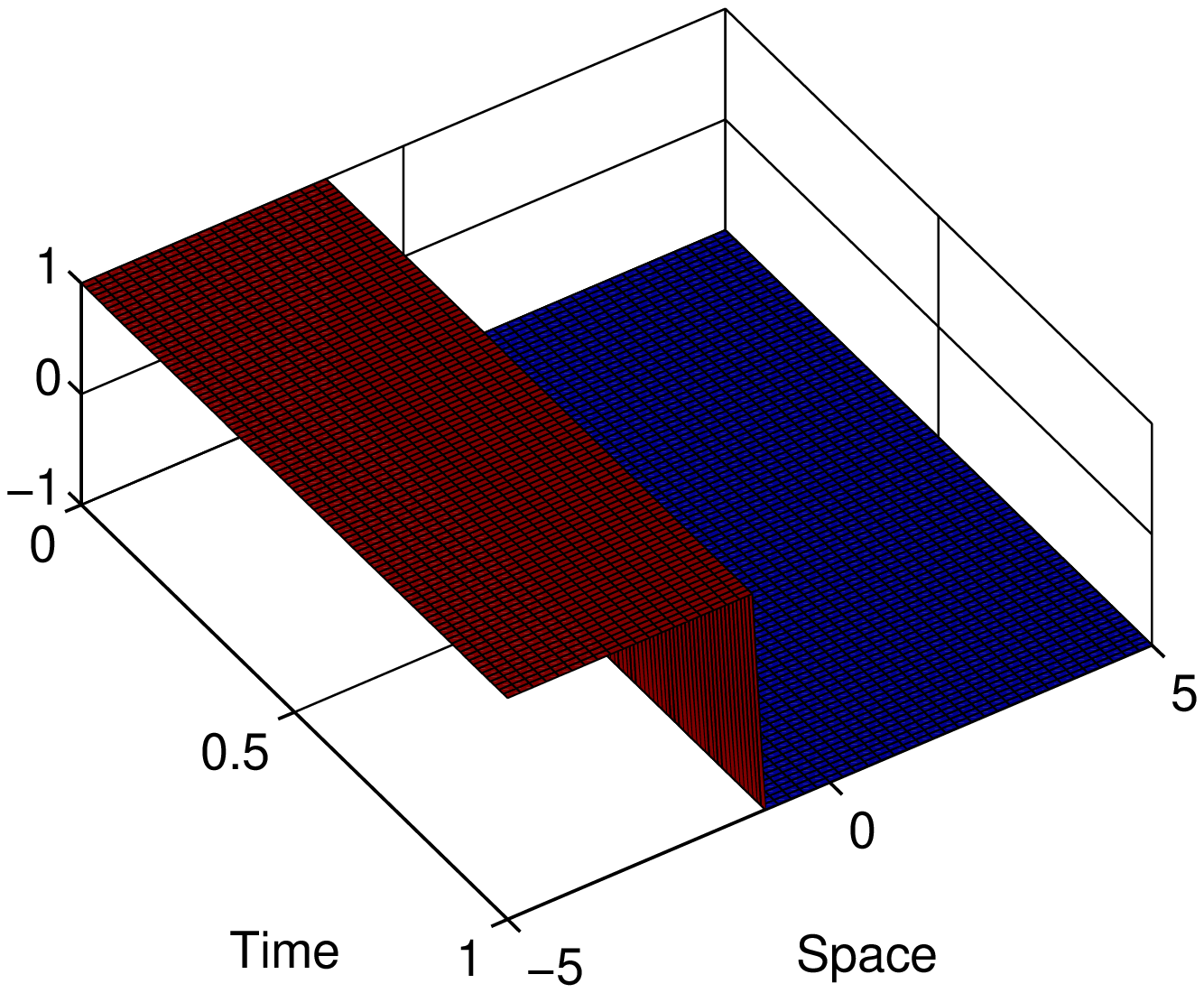}
      \caption{\newline Optimal Control}
      \label{figMixConcave3}
      \vspace{4mm}
   \end{minipage}
   }
\end{figure}

\begin{figure}[htb]
\centering
\begin{tabular}{|c||cc|ccc||cc|} \hline
 & \multicolumn{5}{c||}{\text{Test Case 2}} & \multicolumn{2}{c|}{\text{Test Case 3}} \\ \hline
 & \multicolumn{2}{c|}{\text{Algorithm 1}} & \multicolumn{3}{c||}{\text{Algorithm 2}} & \multicolumn{2}{c|}{\text{Algorithm 1}}\\ \hline
\text{$\ell$} & \text{$\chi(m^\ell)$} & \text{$\varepsilon_\ell$} & \text{$\chi(m^\ell)$} & \text{$\varepsilon_\ell$} & \text{\phantom{$\big|^l$}$q$\phantom{$\big|^l$}} & \text{$\chi(m^\ell)$} & \text{$\varepsilon_\ell$} \\ \hline
0   & -2      & 2,1           & -2      & 2,1           & 0  & 2      & 8,6.$10^{-1}$  \\
2   & -3,4572 & 1,0.$10^{-1}$ & -2,0459 & 2,0           & 2  & 0,7774 & 1,5.$10^{-2}$  \\
4   & -3,5109 & 2,6.$10^{-2}$ & -2,2491 & 1,6           & 4  & 0,7430 & 1,9.$10^{-3}$  \\
6   & -3,5323 & 2,5.$10^{-3}$ & -3,1521 & 4,0.$10^{-1}$ & 6  & 0,7390 & 2,4.$10^{-4}$  \\
8   & -3,5343 & 2,3.$10^{-4}$ & -3,5163 & 1,8.$10^{-2}$ & 8  & 0,7385 & 1,37.$10^{-4}$ \\
10  & -3,5346 & $\approx 0  $ & -3,5334 & 1,1.$10^{-3}$ & 10 & 0,7384 & $\approx 0$    \\
12  & -       & -             & -3.5345 & 3,9.$10^{-5}$ & 12 & -      & -              \\
14  & -       & -             & -3.5346 & 3,1.$10^{-9}$ & 14 & -      & -              \\ \hline
\end{tabular}
\caption{Convergence results of test cases 2 and 3}
\label{figConvResultsTest23}
\end{figure}

\begin{figure}[htb]
\fbox{
\begin{minipage}[c]{0.95\linewidth}
\vspace{3mm}
\centering{\textbf{Test Case 4}}
\vspace{3mm}

   \begin{minipage}[c]{.45\linewidth}
   
\vspace{2mm}   
   
   	  \includegraphics[width=\textwidth]{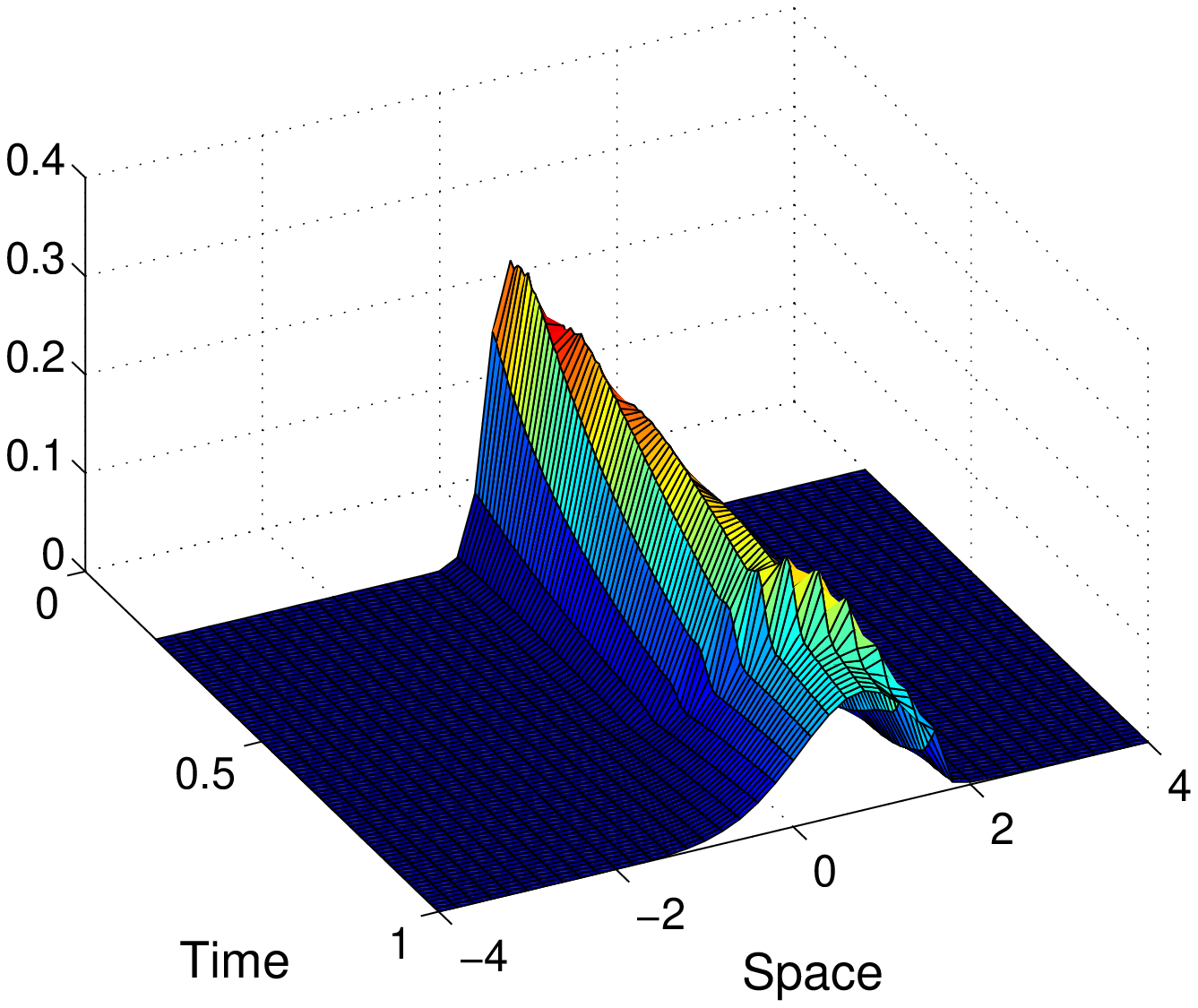} 
      \caption{\newline Regularized \newline distribution}
      \label{figCVaR1}
      \vspace{9.5mm}
      \includegraphics[width=\textwidth]{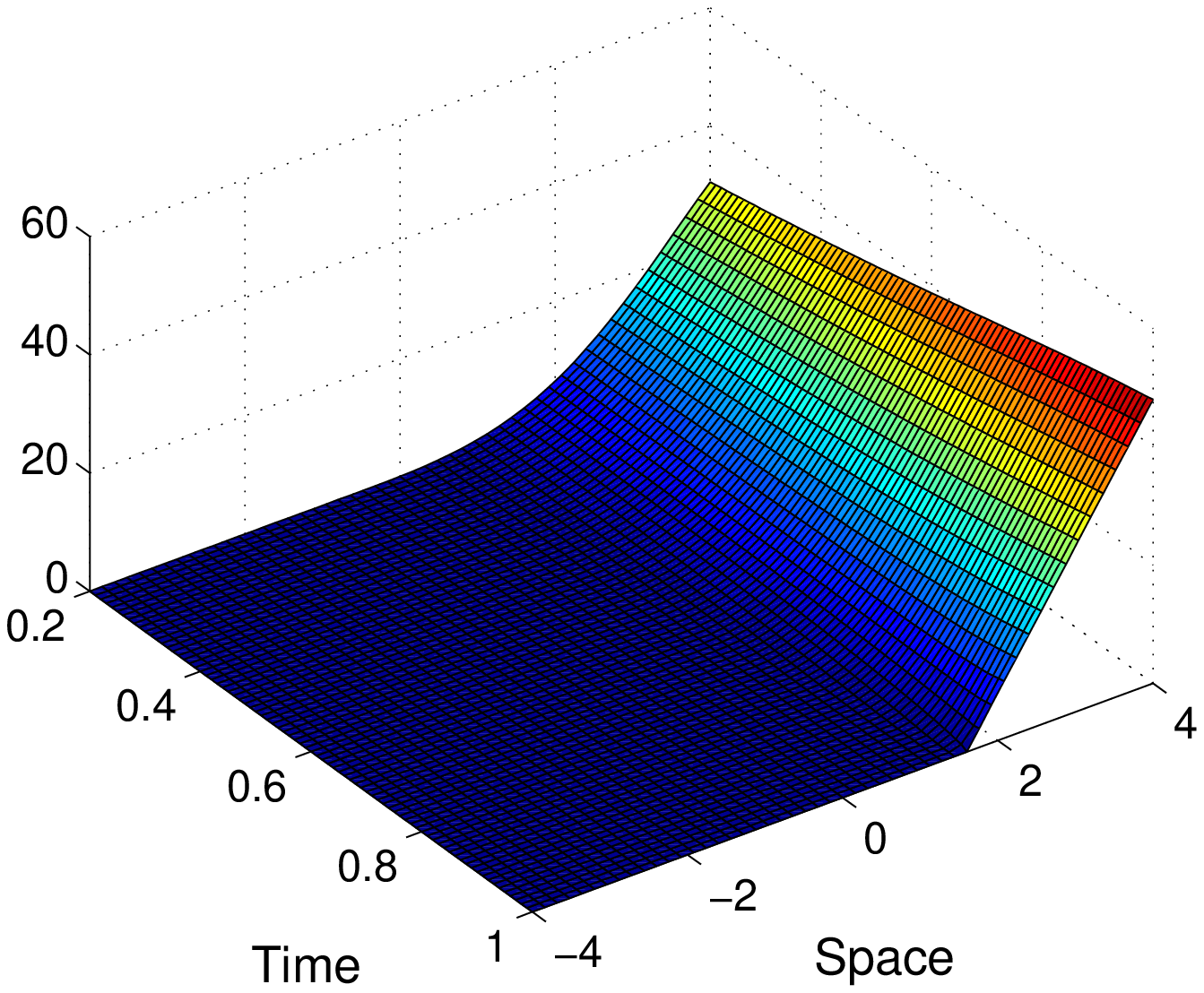}
      \caption{\newline Value function}
      \label{figCVaR2}
      \vspace{9mm}
   \end{minipage} \hfill
   \begin{minipage}[c]{.45\linewidth}
      \includegraphics[width=\textwidth]{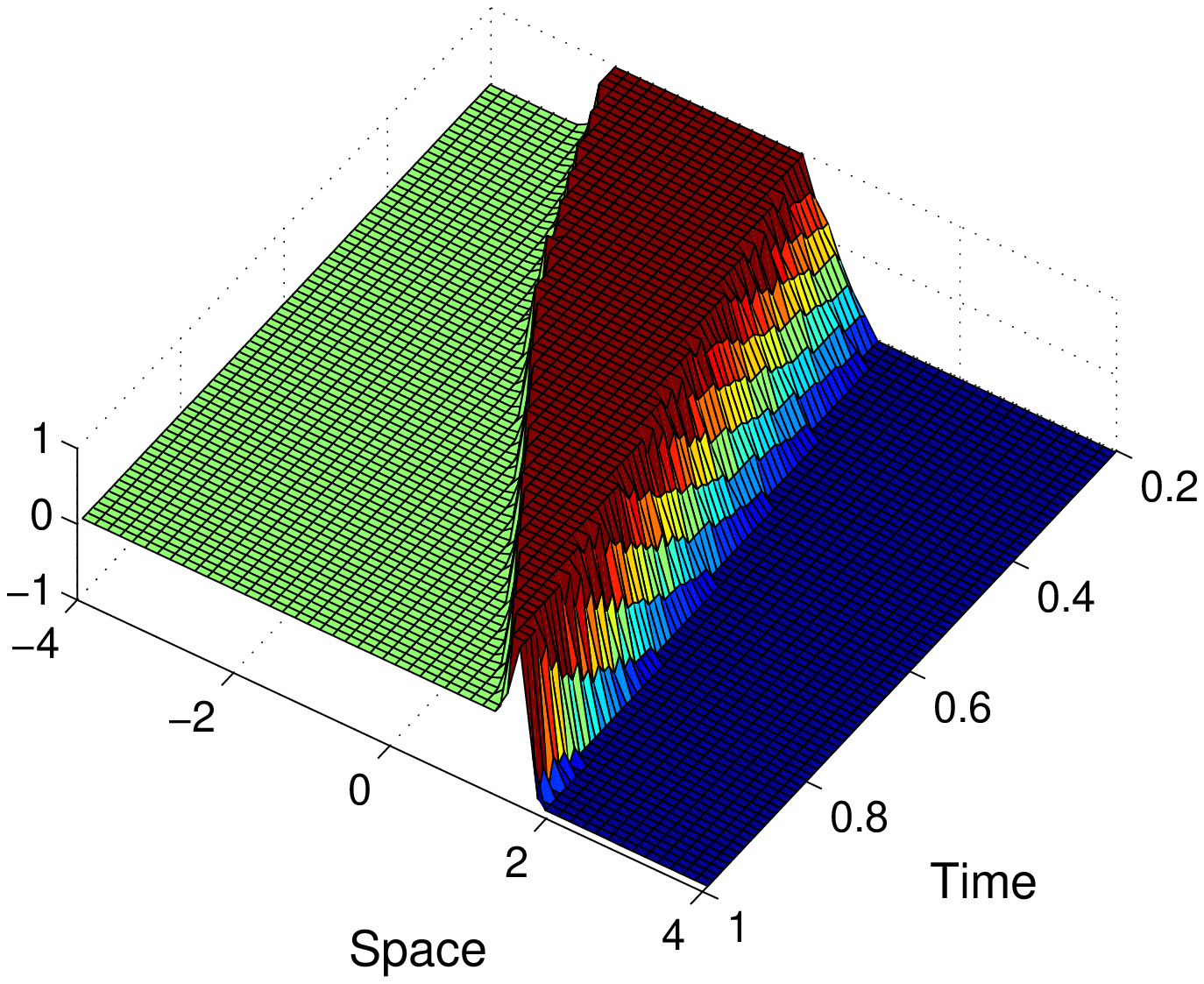} 
      \caption{\newline Optimal control}
      \label{figCVaR3}
      
      \vspace{14mm}
      

\centering
\begin{tabular}{|c||cc|} \hline
     & \multicolumn{2}{c|}{\text{Algorithm 1}} \\ \hline
$\ell$ & $\chi(m^\ell)$ & \phantom{$\big|^l$} $\varepsilon_\ell$ \phantom{$\big|^l$} \\ \hline
0  & 2,0545 & 2,5.$10^{-1}$ \\
5  & 1,8972 & 1,0.$10^{-1}$ \\
10 & 1,8051 & 9,0.$10^{-3}$ \\
15 & 1,7961 & 3,6.$10^{-5}$ \\
20 & 1,7961 & 9,6.$10^{-7}$ \\
25 & 1,7961 & 1,1.$10^{-8}$ \\
30 & 1,7961 & 1,0.$10^{-10}$ \\ \hline
\end{tabular}

\includegraphics[width= 0\textwidth]{Figures/MixConcave/probDistrib.eps}
\caption{\newline Convergence \newline results}
      \label{figCVaR4}

\vspace{2mm}

   \end{minipage}
   \end{minipage}
   }
\end{figure}

We present numerical results for four different academic problems. The considered controlled SDEs are the following:
\begin{equation*}
\begin{array}{llll}
\text{Test cases 1, 2, 3:} & \dd X_t= u_t \dd t + \dd W_t, & X_0= 0, & U=[-1,1] \\
\text{Test case 4: } & \dd X_t= u_t \dd t + (1-u_t) \dd W_t, & X_0= 0, & U= [-1,1]. 
\end{array}
\end{equation*}
The chosen time step is $\delta t= 0.01$, the state space is discretized with $S= \{ -5, -5 + \delta x, -5 + 2\delta x, ..., 5 \}$, where $\delta x= 0.01$. The set of feasible controls is also discretized, with $\{ -1, -1 + \delta u, -1 + 2 \delta u, ... , 1 \}$, where $\delta u = 0.05$ (the minimization problem in \eqref{eqDynProgForScheme} is solved by enumeration).

As we already mentioned in remark \ref{remAlgo}, the discretized probability distribution obtained with a the semi-Lagrangian scheme has an oscillatory behaviour (see figure \ref{figTransport1}). However, in the three considered examples, it is easy to build a ``regularized" probability distribution. The approach used is the following: we denote by $S'$ the following subset of $S$: $S'= \{ -5, -5+ \delta y, -5 +2 \delta y,..., 5 \}$, with $\delta y= 0.2$. Given a probability distribution $m_j$ on $S$ at time $j$, we compute a regularized distribution $\tilde{m}_j$ on $S'$ as follows:
\begin{equation} \label{eqRegularization}
\tilde{m}_j(y)= \sum_{\begin{subarray}{c} x \in S \\ |y-x| \leq \delta y \end{subarray}} \frac{\delta y - | y-x |}{\delta y} m_j(x), \quad \forall x \in S.
\end{equation}
One can easily check that $d_1(m_j,\tilde{m}_j) \leq \delta y/2$.

The graphs of the regularized probability distributions in the three studied test cases (figures \ref{figTransport2}, \ref{figMixConvexe1}, and \ref{figMixConcave1}) are in our opinion good representations.
We do not pretend to justify here the discretization used for our problem. However, we think that the (unregularized) probability distribution obtained with our numerical scheme is acceptable (despite its oscillatory behavior), in so far as the regularization provides a good representation and is close in the $d_1$-norm, the cost function being continuous for this norm.

The time needed for a backward pass is approximately $0.23$s, the time needed for a forward pass is approximately $0.05$s.

\paragraph{Test case 1: Wasserstein distance}

We test cost function: $\chi(m)= d_2(m,\bar{m})$, where:
\begin{equation*}
\bar{m}= \frac{1}{3} \big( \delta_{-2} + \delta_{0} + \delta_{2} \big).
\end{equation*}
In dimension 1, this cost can be easily computed, as well as a sub-gradient which we use as if it was a gradient.
In Figure \ref{figCaseTransportConvergence}, we show the value of the cost function at different iterations $\ell$, as well as the criterion $\varepsilon_{\ell}$, for the two algorithms. For the second algorithm, we also show the number of backward and forward passes $q$.

The figures \ref{figTransport1}, \ref{figTransport2}, \ref{figTransport3}, and \ref{figTransport4} (page \pageref{pageTestCase1}) show respectively the probability distribution, the regularized probability distribution (obtained with \eqref{eqRegularization}), the value function, and the optimal control that we obtain after a large number of iterations of the second algorithm.

Let us comment on the form of the value function and the optimal solution. The choice of the probability distribution $\bar{m}$ has the following effect: one tries to attract the system at one of the three points: $-2$, $0$, and $2$. These three points are the three local minimizers of the dual variable at the final time.
At the final time, the optimal control is bang-bang and has 5 discontinuity points. Three of them are the three attractors ($-2$, $0$, and $2$): the optimal control is equal to 1 on the left and to -1 on the right, at each of these points. The two other discontinuity points are the two local maximizers of the dual variable. Due to diffusion, the value function has only two local minimizers for early times. At these times, the optimal control has only three discontinuity points and only the points $-2$ and $2$ play a role of attractor. 

We observed that the criterion $\varepsilon_{\ell}$ does not seem to converge (even after a very large number of iterations). The cost function is probably not differentiable at the optimal solution (which happens if the dual problem has several optimal solutions). In this case, the derivative is discontinuous at the optimal solution, which prevents the criterion $\varepsilon_\ell$ from converging.
Comparing algorithm 2 with algorithm 1, we observe that algorithm 2 is particularly efficient. The difference of costs of the two methods is negligible.

\paragraph{Test cases 2 and 3: combination of expectation and standard deviation}

In the test case 2 (resp.\@ test case 3), we use the following cost function:
\begin{align*}
\chi(m)= \ & \mathbb{E} \big[ X_T \big] + \beta \sqrt{\text{Var} \big[ X_T \big] } \\
= \ & \int_{\R} x \dd m(x) + \beta \Big( \int_{\R} x^2 \dd m(x) - \Big( \int_{\R} x \dd m(x) \Big)^2 \Big)^{1/2},
\end{align*}
with $\beta= -2$ (resp.\@ $\beta= 2$). One can easily check that in the test case 2, the cost function is convex, whereas it is concave in the test case 3. The algorithm 1 is therefore sufficient for the test case 3, since then the solution of \eqref{eqStepSize} is $-1$.

In the two test cases, the two algorithms converge quickly, as shown in Figure \ref{figConvResultsTest23}. The difference of costs for the two algorithms is negligible (in the test case 2).
The probability distribution, the value function, and the optimal control are shown page \pageref{figMixConvexe1}.
As a consequence of formula \eqref{eqDiffFunctionExp}, the value function is a parabola at the final in the two cases (a concave one in the test case 2, a convex one in the test case 3). The optimal control is constant in time and has a bang-bang structure: equal to 1 when the value function is decreasing and equal to $-1$ when it is increasing.

\paragraph{Test case 4: Conditional Value at Risk}

The cost function used in this test case is the conditional value at risk, with parameter $\beta= 0,95$. The chosen controlled SDE is slightly different for this example and must be understood as follows: negative controls are efficient, in so far as they induce the strong decrease of the state variable (in expectation). They are also risky, since the volatily is higher. To the contrary, positive control are less risky, but expensive.

Results are presented in Figures \ref{figCVaR1}-\ref{figCVaR4}. The criterion is very small at the end. Note that since the cost function is concave, the first algorithm is sufficient. Let us comment on the obtained graphs. The CVaR focuses on the worst cases: when $X_t$ is high, a risky strategy is employed. To the contrary, when $X_t$ is low, which is a favorable case, the gains are not taken into account and therefore a less risky strategy is prefered.

\paragraph{Conclusion of the numerical results}

We have tested the two methods on four academic examples, for which the cost function is continuous for the Wasserstein distance. A semi-Lagrangian scheme has been used for the discretization. The two proposed methods converge. The controls provided by the second method (which only computes feedback controls) are as good as the controls of the first one (which allows a larger class of control processes). Except in the test case 1, for which the cost function is not continuously differentiable, the criterion $\varepsilon_\ell$ converges to 0. A further observation of the convergence results also shows that the cost $\chi(m^\ell)$ and the criterion $\varepsilon_\ell$ converge at a linear rate in the four cases, for the first method as well as the second one (for test cases 1 and 2).

\section*{Acknowledgements}

The author gratefully acknowledges the Austrian Science Fund (FWF) for financial support under SFB F32 ``Mathematical Optimization and Applications in Biomedical Sciences". The author thanks the three anonymous referees for their valuable comments on the original manuscript.

\bibliographystyle{plain}
\bibliography{biblio}


\end{document}